\documentclass[11pt]{amsart}
\usepackage{a4wide}
\usepackage{cyrillic} 
\newcommand{\cyrrm}{\fontencoding{OT2}\selectfont\textcyrup}
\usepackage{amssymb}
\usepackage{latexsym}
\usepackage{mathrsfs}
\usepackage{xspace}
\usepackage{enumerate}
\newtheorem{prop}{Proposition}[section]
\newtheorem{thm}[prop]{Theorem}
\newtheorem*{thm*}{Theorem}
\newtheorem*{cor*}{Corollary}

\newtheorem*{addendum*}{Addendum}
\newtheorem{cor}[prop]{Corollary}
\newtheorem{lem}[prop]{Lemma}

\newtheorem*{convention*}{Convention}
\newtheorem{question}[prop]{Question}
\newtheorem{ithm}{Theorem}
\newtheorem{icor}[ithm]{Corollary}

\theoremstyle{definition}
\newtheorem*{defn*}{Definition}

\newtheorem{remark}[prop]{Remark}

\newtheorem*{scholium*}{Scholium}
\theoremstyle{remark}
\newtheorem{example}[prop]{Example}
\newtheorem*{example*}{Example}
\numberwithin{equation}{section}

\newcommand{\RR}{\mathbf{R}}

\newcommand{\ZZ}{\mathbf{Z}}

\newcommand{\SL}{\mathrm{SL}}

\newcommand{\la}{\langle}
\newcommand{\ra}{\rangle}
\newcommand{\inv}{^{-1}}

\newcommand{\minus}{\setminus}
\newcommand{\norma}{\mathscr{N}}
\newcommand{\centra}{\mathscr{Z}}

\newcommand{\QZ}{\mathscr{Q\!Z}}
\newcommand{\se}{\subseteq}

\newcommand{\lra}{\longrightarrow}

\newcommand{\Adc}{$\Ad\!$-closed\xspace}

\newcommand{\tto}{\twoheadrightarrow}

\newcommand{\Max}{\mathbf{Max}}
\newcommand{\Maxqf}{\mathbf{Max}_{\mathrm{QF}}}
\newcommand{\Min}{\mathbf{Min}}
\def\bs#1.{
              \def\temp{#1}
              \ifx\temp\empty
                   \mathcal{B}
              \else
                   \mathcal{B}(#1)
              \fi
}
\newcommand{\cat}{{\upshape CAT(0)}\xspace}
\newcommand{\tangle}[2]
{\angle_\mathrm{T}(#1,#2)}
\newcommand{\aangle}[3]
{\angle_{#1}(#2,#3)}
\newcommand{\cangle}[3]
{\overline{\angle}_{#1}(#2,#3)}

\DeclareMathOperator{\Res}{Res}
\DeclareMathOperator{\Ad}{Ad}
\DeclareMathOperator{\Ker}{Ker}

\DeclareMathOperator{\LF}{Rad_{\mathscr{L\!F\!}}}

\def\Aut{\mathop{\mathrm{Aut}}\nolimits}
\def\Inn{\mathop{\mathrm{Ad}}\nolimits}
\def\Linn{\mathop{\mathrm{Linn}}\nolimits}

\def\Out{\mathop{\mathrm{Out}}\nolimits}
\def\Sym{\mathop{\mathrm{Sym}}\nolimits}

\makeindex
\begin{document}
\title[Decomposing locally compact groups into simple
pieces]{Decomposing locally compact groups\\ into simple pieces}
\author[Pierre-Emmanuel Caprace]{Pierre-Emmanuel Caprace*}
\address{UCLouvain, 1348 Louvain-la-Neuve, Belgium}
\email{pe.caprace@uclouvain.be}
\thanks{*F.N.R.S. Research Associate}
\author[Nicolas Monod]{Nicolas Monod$^\ddagger$}
\address{EPFL, 1015 Lausanne, Switzerland}
\email{nicolas.monod@epfl.ch}
\thanks{$^\ddagger$Supported in part by the Swiss National Science
Foundation}
%
\begin{abstract}
We present a contribution to the structure theory of locally compact
groups. The emphasis is on compactly generated locally compact
groups which admit no infinite discrete quotient. It is shown that
such a group possesses a characteristic cocompact subgroup which is
either connected or admits a non-compact non-discrete topologically
simple quotient. We also provide a description of characteristically
simple groups and of groups all of whose proper quotients are
compact. We show that Noetherian locally compact groups without
infinite discrete quotient admit a subnormal series with all
subquotients compact, compactly generated Abelian, or compactly
generated topologically simple.

Two appendices introduce results and examples around the concept of \emph{quasi-product}.
\end{abstract}
\maketitle
\let\languagename\relax  

\tableofcontents

\section{Introduction}
\subsection*{On structure theory}
The structure of \emph{finite} groups can to a large extent be reduced
to finite simple groups and the latter
have famously been classified (see
\emph{e.g.}~\cite{Gorenstein-Lyons-Solomon} \emph{sqq.}).

For general locally compact groups, both the reduction to simple groups
and the study of the latter still constitute
major challenges. The connected case has found a satisfactory answer:
Indeed, the solution to Hilbert's fifth
problem (see~\cite[4.6]{Montgomery-Zippin}) reduces the question to Lie
theory upon discarding compact kernels. Lie
groups are then described by investigating separately the soluble groups
and the simple factors, which are
classified since the time of \'E.~Cartan.

\medskip

The contemporary structure problem therefore regards totally
disconnected groups;
there is not yet even a conjectural picture of a structure theory. In
fact, until recently, the only
structure theorem on totally disconnected locally compact groups was
this single sentence in van~Dantzig's
1931 thesis:
\begin{quotation}
``\emph{Een (gesloten) Cantorsche groep bevat willekeurig kleine open
ondergroepen.}''\hfill(III~\S\,1, TG~38 on page~18 in~\cite{vanDantzig31})
\end{quotation}

\smallskip
Recent progress, including statements on simple groups, is provided
by the work of G.~Willis \cite{Willis94, Willis07}. New examples of
simple groups have appeared in the geometric context of trees and of
general buildings.

\medskip

As results and examples for simple groups are being discovered, it
becomes desirable to have a reduction step to simple
groups~--- in parallel to the known cases of finite and connected
groups. However, any reasonable attempt
at classification must in one way or another exclude \emph{discrete}
groups: the latter are widely considered to be
unclassifiable; this opinion can be given a mathematical content as
\emph{e.g.} in~\cite{Thomas-Velickovic}. The discrete case also illustrates that
there may be no simple (infinite) quotient, or even subquotient, at all.

\smallskip

The following trichotomy shows that, away from the unavoidable discrete situation,
there is a compelling first answer.

\begin{ithm}\label{thm:structure:first}
Let $G$ be a compactly generated locally compact group. Then exactly one of the
following holds.
\begin{enumerate}

\item $G$ has an infinite discrete quotient.

\item $G$ has a cocompact normal subgroup that is connected and soluble.

\item $G$ has a cocompact normal subgroup that admits exactly $n$
non-compact simple quotients (and no non-trivial discrete quotient), where $0<n<\infty$.
\end{enumerate}
\end{ithm}

By a \emph{simple} group, we mean a \emph{topologically simple}
group, \emph{i.e.} a group all of whose Hausdorff quotients are
trivial. Since a cocompact closed subgroup of a compactly generated
locally compact group is itself compactly
generated~\cite{Macbeath-Swierczkowski59},
it follows that the $n$ simple quotients
appearing in (iii) of Theorem~\ref{thm:structure:first} are
compactly generated. (We always implicitly endow quotient groups with the quotient topology.)

\medskip
The above theorem describes the \emph{upper structure} of $G$. The first
alternative can be made more precise in combination with the
well-known (and easy to establish) fact that an infinite finitely
generated group either admits an infinite residually finite quotient
or has a finite index subgroup which admits an infinite simple
quotient. In some sense, Theorem~\ref{thm:structure:first} plays the
r\^ole of a non-discrete analogue of the latter fact; notice however
that the finiteness of the number $n$ of simple sub-quotients in
case (iii) above is  particular to non-discrete groups. It turns out
that the above theorem is supplemented by the following
description of the \emph{lower structure} of $G$, which does not seem to
have any analogue in the discrete case.

\begin{ithm}\label{thm:structure:lower}
Let $G$ be a compactly generated locally compact group. Then one of
the following holds.
\begin{enumerate}
\item $G$ has an infinite discrete normal subgroup.

\item $G$ has a non-trivial closed normal subgroup which is
compact-by-\{connected soluble\}.

\item $G$ has exactly $n$ non-trivial minimal closed normal subgroups,
where $0<n<\infty$.
\end{enumerate}
\end{ithm}

The normal subgroups appearing in the above have no reason to be
compactly generated in general. On another hand, since any
(Hausdorff) quotient of a compactly generated group is itself
compactly generated, it follows that
Theorem~\ref{thm:structure:lower} may be applied repeatedly to the
successive quotients that it provides. Such a process will of course
not terminate after finitely many steps in general. However, if $G$
satisfies additional finiteness conditions, this recursive process
may indeed reach an end in finite time. In order to make this
precise, we introduce the following terminology. We call a topological group
$G$ \textbf{Noetherian}\index{Noetherian group} if it satisfies the ascending
chain condition on open subgroups. Obvious examples are provided by
compact (\emph{e.g.} finite) groups, connected groups and polycyclic
groups. If $G$ is
locally compact, then $G$ is Noetherian if and only if every open
subgroup is compactly generated. (Warning: the notion introduced
in~\cite[\S\,III]{Guivarch73}
is more restrictive as it posits compact generation of all \emph{closed}
subgroups.)

\begin{ithm}\label{thm:CompositionSeries}
Let $G$ be a locally compact Noetherian group. Then $G$ possesses an
open normal subgroup $G_k$ and a finite series of closed subnormal subgroups
$$1 = G_0 \lhd G_1 \lhd G_2 \lhd \cdots \lhd G_k \lhd G$$
such that, for each $i \leq k$, the
subquotient $G_i/G_{i-1}$ is either compact, or isomorphic to $\ZZ$
or $\RR$, or topologically simple non-discrete and compactly generated.
\end{ithm}

In the special case of connected groups, the existence of the above
decomposition follows easily from the solution to Hilbert's fifth
problem. In that case, the simple subquotients are connected
non-compact adjoint simple Lie group, while the presence of discrete
free Abelian groups accounts for possible central extensions of
simple Lie groups such as $\widetilde{\mathrm{SL}_2(\RR)}$ (an example
going back
to Schreier~\cite[\S\,5 Beispiel~2]{Schreier25}).
No analogue of Theorem~\ref{thm:CompositionSeries} seems to be known
for \emph{discrete} Noetherian groups.

\subsection*{Characteristically simple groups and quasi-products}

By a \textbf{characteristic subgroup}\index{characteristic subgroup}
of a topological group $G$, we
mean a closed subgroup which is preserved by every topological group
automorphism of $G$. A group admitting no non-trivial such subgroup is
called \textbf{characteristically simple}\index{characteristically simple}.
This property is satisfied by any \emph{minimal}
normal subgroup, for instance those occurring in
Theorem~\ref{thm:structure:lower}.

In fact, our above results lead also to a description of
characteristically simple groups,
as follows.

\begin{icor}\label{cor:CharSimple}
Let $G$ be a compactly generated locally compact group. If $G$ is
characteristically simple, then one of the following  holds.
\begin{enumerate}
\item $G$ is discrete.

\item $G$ is compact.

\item $G \cong \RR^n$ for some $n$.

\item $G$ is a quasi-product with topologically simple pairwise
isomorphic quasi-factors.
\end{enumerate}
\end{icor}

By definition, a topological group $G$ is called a
\textbf{quasi-product}\index{quasi-product} with
\textbf{quasi-factors}\index{quasi-factor!of a quasi-product}
$N_1,  \dots, N_p$ if $N_i$ are closed normal subgroups such that the
multiplication map
$$N_1\times\cdots\times N_p \lra G$$
is injective with dense image. Usual direct products are obvious
examples, but the situation is much more complicated for general totally
disconnected groups. The above definition degenerates in the commutative
case; for instance, $\RR$ is a quasi-product with quasi-factors $\ZZ$
and $\sqrt 2\ZZ$. Several \emph{centrefree} examples of quasi-products,
including characteristically simple ones, will be constructed in
Appendix~\ref{app:B}. However, as of today we are not aware of any
\emph{compactly generated} characteristically simple group which falls
into Case~(iv) of Corollary~\ref{cor:CharSimple} without being a genuine
direct product. As we shall see in Appendix~\ref{app:B}, the existence
of such an example is equivalent to the existence of a compactly
generated topologically simple locally compact group admitting a proper
dense normal subgroup.

\subsection*{Groups with every proper quotient compact}
The first goal that we shall pursue in this article is to describe
the compactly generated locally compact groups which admit only
compact proper quotients. The non-compact groups satisfying this
condition are sometimes called \textbf{just-non-compact}\index{just-non-compact}.
In the discrete case, the corresponding notion is that of
\textbf{just-infinite groups}\index{just-infinite}, namely discrete groups all of whose
proper quotients are finite. A description of these was given by
J.~S.~Wilson in a classical article (\cite{Wilson71},
Proposition~1). Anticipating on the terminology introduced below, we
can epitomise our contribution to this question as follows:

\medskip
\itshape A just-non-compact group is either discrete or monolithic.
\upshape
\medskip

The most obvious case where a topological group has only compact
quotients is when it is \textbf{quasi-simple}\index{quasi-simple},
which means that it possesses a cocompact normal subgroup which is
topologically simple and contained in every
non-trivial closed normal subgroup.

This situation extends readily to the following. We say that a
topological group is \textbf{monolithic}\index{monolithic} with
\textbf{monolith}~$L$\index{monolith}
if the intersection of all non-trivial closed normal subgroups is
itself a \emph{non-trivial} group~$L$. Non-quasi-simple examples are
provided by the standard wreath product of a topologically simple
group by a finite transitive permutation group (see Construction~1
in~\cite{Wilson71}). Notice that the monolith is necessarily
characteristically simple.

\bigskip

In the discrete case, groups with only finite proper quotients can be
very far from monolithic, indeed
residually finite: examples are provided by all lattices in connected
centrefree simple Lie groups of rank at
least two in view of a fundamental theorem of
G.~Margulis~\cite{Margulis}; for instance, $\mathbf{PSL}_3(\ZZ)$
(this particular case was already known to J.~Mennicke~\cite{Mennicke65}).
The following result shows that such examples do not exist in the
non-discrete case.

\begin{ithm}\label{thm:monolithic}
Let $G$ be a compactly generated non-compact locally compact group such
that every
non-trivial closed normal subgroup is cocompact. Then one of the
following holds.
\begin{enumerate}
\item $G$ is monolithic and its monolith is a quasi-product with
finitely many isomorphic topologically simple groups as quasi-factors.

\item $G$ is monolithic with monolith~$L \cong \RR^n$. Moreover $G/L$ is
isomorphic to a closed irreducible
subgroup of $\mathbf{O}(n)$. In particular $G$ is an almost connected
Lie group.

\item $G$ is discrete and residually finite.
\end{enumerate}
\end{ithm}

We shall see in \S\,\ref{sec:BN} below that this theorem yields a
topological simplicity corollary for locally
compact groups endowed with a $BN$-pair which supplements classical
results by J.~Tits~\cite{Ti64}.

\medskip

The proof of Theorem~\ref{thm:monolithic} relies on an analysis of
filtering families of closed normal subgroups in totally
disconnected locally compact groups, which is carried out in
Proposition~\ref{prop:FilterNormalSubgroups} below. As a by-product,
it yields in particular a characterisation of residually discrete
groups (see Corollary~\ref{cor:ResDiscrete}) and the following
easier companion to Theorem~\ref{thm:monolithic}, where $\Res(G)$
denotes the \textbf{discrete residual}\index{discrete residual} of $G$, namely the
intersection of all open normal subgroups.

\begin{ithm}\label{thm:OpenNormal}
Let $G$ be a compactly generated locally compact group all of whose
discrete quotients are finite.
Then $\Res(G)$ is a cocompact characteristic closed subgroup of $G$
without non-trivial discrete quotient.
\end{ithm}

The cocompact subgroup $\Res(G)$ is compactly generated (see
\cite{Macbeath-Swierczkowski59} or Lemma~\ref{lem:cocompact} below).
Since any compact totally
disconnected group is a profinite group and, hence, admits numerous
discrete quotients, it follows that, under
the hypotheses of Theorem~\ref{thm:OpenNormal},  \emph{any compact
quotient of the discrete residual is
connected}. Loosely speaking, the discrete residual is thus a sort of
\emph{cocompact core} of the group $G$. As
a consequence, we obtain the following.

\begin{icor}\label{cor:MaxCptQuotient}
Let $G$ be a compactly generated totally disconnected locally
compact group all of whose discrete quotients are finite. Then the
discrete residual of $G$ is cocompact and admits no non-trivial
discrete or compact quotient. \qed
\end{icor}

Theorem~\ref{thm:structure:first} follows easily by combining
Theorem~\ref{thm:monolithic} with the fact, due to R.~Grigorchuk and
G.~Willis, that any non-compact compactly generated locally compact
group admits a just-non-compact quotient (unpublished, see
Proposition~\ref{prop:GriWi} below).

\subsection*{Acknowledgements}
We are very grateful for the detailed comments provided by the anonymous
referee.


\section{Basic tools}

\subsection*{Generalities on locally compact groups}

In this preliminary section, we collect a number of subsidiary facts
on locally compact groups which will be used repeatedly in the
sequel. Unless specified otherwise, all topological groups are assumed Hausdorff.

We will frequently invoke the following well-known statement
without explicit reference.

\begin{lem}\label{lem:cocompact}
If a closed subgroup of a compactly generated locally compact group is
cocompact, then it is itself compactly generated.
\end{lem}

\begin{proof}
See~\cite{Macbeath-Swierczkowski59}, Corollary~2.
\end{proof}

We say that a subgroup $H$ of a topological group $G$ is
\textbf{topologically locally finite}\index{topologically locally finite} if every finite
subset of $H$ is contained in a compact subgroup of $G$. Any locally
compact group $G$ possesses a maximal
normal topologically locally finite subgroup which is closed and called
the \textbf{LF-radical}\index{LF-radical}\index{radical!LF} and denoted $\LF(G)$; another
important fact is that any compact subset of a locally compact
topologically locally finite group is contained
in a compact subgroup. We refer to~\cite{Platonov}
and~\cite[\S\,2]{CapraceTD} for details.

It is well known that the LF-radical is compact for connected groups:

\begin{lem}\label{lem:LF:connected}
Every connected locally compact group admits a maximal compact normal
subgroup.
Moreover, the corresponding quotient is a connected Lie group.
\end{lem}

\begin{proof}
The solution to Hilbert's fifth
problem~\cite[Theorem~4.6]{Montgomery-Zippin} provides a compact normal
subgroup
such that the quotient is a Lie group; now the statement follows from
the corresponding fact for connected Lie groups.
\end{proof}

As a further element of terminology, the \textbf{quasi-centre}\index{quasi-centre} of a
topological group $G$ is the subset $\QZ(G)$ consisting of all
those elements possessing an open centraliser. The subgroup $\QZ(G)$
is topologically characteristic in $G$, but need not be closed.
Since any element with a discrete conjugacy class possesses an open
centraliser, it follows that the quasi-centre contains all discrete
normal subgroups of $G$.

\medskip

We shall use the following result of U{\v{s}}akov, for which we recall that a topological group is called
\textbf{topologically $\overline{\text{FC}}$}\index{topologically FC-group@topologically $\overline{\text{FC}}$-group}\index{FC-group!topologically}
if every conjugacy class has compact closure.

\begin{thm}[U{\v{s}}akov~\cite{Ushakov}]\label{thm:Ushakov}
Let $G$ be a locally compact topologically $\overline{\text{FC}}$-group. Then the 
union of all compact subgroups of $G$ forms a closed normal subgroup, which therefore coincides with
$\LF(G)$, and the corresponding quotient $G/\LF(G)$ is Abelian.

Moreover, if in addition $G$ is compactly generated, then $G$ is compact-by-Abelian.
More precisely, $\LF(G)$ is compact and $G/\LF(G)$ is isomorphic to $\RR^n\times \ZZ^m$ for some $n,m\geq 0$.\qed
\end{thm}

The following consideration provides a necessary (and sufficient)
condition for the group considered in Theorem~\ref{thm:monolithic}
to admit a non-trivial \emph{discrete} normal subgroup.

\begin{prop}\label{prop:discrete_normal}
Let $G$ be a compactly generated non-compact locally compact group such
that every
non-trivial closed normal subgroup is cocompact.
If $G$ admits a non-trivial discrete normal subgroup, then $G$ is either
discrete or $\RR^n$-by-finite.
\end{prop}

\begin{proof}
Let $H\lhd G$ be a non-trivial discrete normal subgroup.
Then $H$ is cocompact and, hence, it is a cocompact lattice in the
compactly generated group $G$. Therefore
$H$ is finitely generated. We deduce that the normal subgroup
$\centra_G(H)\lhd G$ is open, since every element of
$H$ has a discrete conjugacy class and, hence, an open centraliser. In
particular, $G$ is discrete if $\centra_G(H)$ is trivial
and we can therefore assume that $\centra_G(H)$ is cocompact.

The quotient $\centra_G(H)/\centra(H)$ is compact since it sits as open
(hence closed) subgroup in $G/H$.
Hence, the centre of $\centra_G(H)$ is cocompact for it contains
$\centra(H)$. Thus Theorem~\ref{thm:Ushakov}
applies to $\centra_G(H)$. We claim that the LF-radical of
$\centra_G(H)$ is trivial: otherwise, being normal in $G$,
it would be cocompact, hence compactly generated, thus compact
(Lemma~2.3 in~\cite{CapraceTD}) and now finally trivial after all since
$G$ is non-compact by hypothesis. In conclusion,
$\centra_G(H)$ is isomorphic to $\RR^n\times \ZZ^m$. In addition, $n$ or
$m$ vanishes since the identity component of
$\centra_G(H)$ is normal in $G$ and hence trivial or cocompact. We
conclude by recalling that $\centra_G(H)$ is open
and cocompact, thus of finite index in $G$.
\end{proof}

\subsection*{Filters of closed normal subgroups}%
We continue with another general fact on compactly generated groups,
which was the starting point of this work. The argument was inspired
by a reading of Lemma~1.4.1 in~\cite{Burger-Mozes1}; it also plays a
key r\^{o}le in the structure theory of isometry groups of
non-positively curved spaces developed
in~\cite{Caprace-Monod_structure}.

\begin{prop}\label{prop:FilterNormalSubgroups}
Let $G$ be a compactly generated totally disconnected locally
compact group. Then any identity neighbourhood $V$ contains a
compact normal subgroup $Q_V$ such that any filtering family of
non-discrete closed normal subgroups of the quotient $G/Q_V$ has
non-trivial intersection.
\end{prop}

\begin{proof}
Let $\mathfrak{g}$ be a \textbf{Schreier graph}\index{Schreier graph} for $G$ associated
to a compact open subgroup $U<G$ contained in the given identity
neighbourhood $V$ (which exists by a classical result in
D.~van~Dantzig's 1931
thesis~\cite[III \S\,1, TG~38 page~18]{vanDantzig31},
see~\cite[III \S\,4 No~6]{BourbakiTGI}). We recall that $\mathfrak{g}$
is obtained by defining the vertex set as $G/U$ and drawing edges
according to a compact generating set which is a union of double
cosets modulo $U$; see~\cite[\S\,11.3]{Monod_LN}. The kernel of the
$G$ action on $\mathfrak{g}$ is nothing but $Q_V = \bigcap_{g \in G}
g U g\inv$ which is compact and contained in $V$.

Since we are interested in closed normal subgroups of the quotient
$G/Q_V$, there is no loss of generality in assuming $Q_V$ trivial.
In other words, we assume henceforth that $G$ acts faithfully on
$\mathfrak g$. Let $v_0$ be a vertex of $\mathfrak g$ and denote by
$v_0^\perp$ the set of neighbouring vertices. Since $G$ is
vertex-transitive on $\mathfrak{g}$, it follows that for any normal
subgroup $N \lhd G$, the $N_{v_0}$-action on $v_0^\perp$ defines a
finite permutation group $F_N < \Sym(v_0^\perp)$ which, as an
abstract permutation group, is independent of the choice of $v_0$.
Therefore, if $N$ is non-discrete, this permutation group $F_N$ has
to be non-trivial since $U$ is open and $\mathfrak g$ connected. Now
a filtering family $\mathscr{F}$ of non-discrete normal subgroups
yields a filtering family of non-trivial finite subgroups of
$\Sym(v_0^\perp)$. Thus the intersection of these finite groups is
non-trivial. Let $g$ be a non-trivial element in this intersection.
For any $N \in \mathscr{F}$, let $N_g$ be the inverse image of
$\{g\}$ in $N_{v_0}$. Thus $N_g$ is a non-empty compact subset of
$N$ for each $N \in \mathscr{F}$. Since the family $\mathscr{F}$ is
filtering, so are $\{N_{v_0} \; | \; N \in \mathscr{F} \}$ and
$\{N_g \; | \; N \in \mathscr{F} \}$. The result follows, since a
filtering family of non-empty closed subsets of the compact set
$G_{v_0}$ has a non-empty intersection.
\end{proof}

\subsection*{Minimal normal subgroups}

With Proposition~\ref{prop:FilterNormalSubgroups} at hand, we deduce
the following. An analogous result for the upper structure of
totally disconnected groups will be established in
Section~\ref{sec:structure} below (see
Proposition~\ref{prop:MaximalNormal}).

\begin{prop}\label{prop:MinimalNormal}
Let $G$ be a compactly generated totally disconnected locally
compact group which possesses no non-trivial compact or discrete
normal subgroup. Then every non-trivial closed normal subgroup of
$G$ contains a minimal one, and the set $\mathscr M$ of non-trivial
minimal closed normal subgroups is finite. Furthermore, for any
proper subset $\mathscr E \subset\mathscr M$, the subgroup
$\overline{\la M \; | \; M \in \mathscr E\ra}$ is properly contained
in $G$.
\end{prop}

\begin{proof}
In view of Proposition~\ref{prop:FilterNormalSubgroups}, any chain
of non-trivial closed normal subgroups of $G$ has a minimal element.
Thus Zorn's lemma ensures that the set $\mathscr{M}$ of minimal
non-trivial closed normal subgroups of $G$ is non-empty, and that
any non-trivial closed normal subgroup contains an element of
$\mathscr M$.

In order to establish that $\mathscr M$ is finite, we use the
following notation. For each subset $\mathscr{E} \se \mathscr{M}$ we
set $M_\mathscr{E} = \la M \; | \; M \in \mathscr{E} \ra$;
$\overline M_\mathscr{E}$ denotes its closure. The following
argument was inspired by the proof of Proposition~1.5.1
in~\cite{Burger-Mozes1}.

We claim that if $\mathscr{E}$ is a proper subset of $ \mathscr{M}$,
then $\overline M_\mathscr{E}$ is a proper subgroup of $G$. Indeed,
for all $M \in \mathscr{E}$ and $M' \in \mathscr{M} \minus
\mathscr{E}$ we have $[M, M'] \se M \cap M' = 1$. Thus $M'$
centralises $\overline M_\mathscr{E}$. In particular, if $\overline
M_\mathscr{E} = G$, then $M'$ centralises $G$. Thus $M' \leq
\centra(G)$ is Abelian, and any proper subgroup of $M'$ is normal in
$G$. Since $M'$ is a minimal normal subgroup, it follows that $M'$
has no proper closed subgroup. The only totally disconnected locally
compact groups with this property being the cyclic groups of prime
order, we deduce that $M'$ is finite, which contradicts the
hypotheses. The claim stands proven.

Consider now the family
$$\mathscr{N} = \{\overline M_{\mathscr{M}\minus{\mathscr{F}}} \; | \;
\mathscr{F} \se \mathscr{M} \text{ is finite}\}$$
of closed normal subgroups of $G$. This family is filtering.
Furthermore the above claim shows that $\overline
M_{\mathscr{M}\minus{\mathscr{F}}}$ is properly contained in $G$
whenever $\mathscr{F}$ is non-empty. Since $\bigcap \mathscr{N} =
1$, it follows from Proposition~\ref{prop:FilterNormalSubgroups}
that $\mathscr{N}$ is finite, and hence $\mathscr{M}$ is so, as
desired.
\end{proof}

\subsection*{Just-non-compact groups}

\begin{proof}[Proof of Theorem~\ref{thm:monolithic}]
We shall use repeatedly the fact that every normal subgroup of $G$ has
trivial LF-radical, which is established as in the above proof
of Proposition~\ref{prop:discrete_normal}. In particular, normal
subgroups of $G$ have no non-trivial compact normal subgroups.

\smallskip

We begin by treating the case where $G$ is totally disconnected.
Let $L$ be the intersection of all
non-trivial closed normal subgroups. We distinguish two cases.

If $L$ is trivial, then Proposition~\ref{prop:FilterNormalSubgroups}
shows that $G$ admits a non-trivial discrete normal subgroup.
Thus Proposition~\ref{prop:discrete_normal} applies and $G$ is discrete;
in that case,
Wilson's result (\cite{Wilson71}, Proposition~1) completes the proof.

\smallskip

Assume now that $L$ is not trivial. Then it is cocompact whence
compactly generated since $G$ is so. Notice that
by definition $L$ is characteristically simple.
We further distinguish two cases.

On the one hand, assume that the quasi-centre $\QZ(L)$ is
non-trivial. Then it is dense in $L$. Since $L$ is compactly
generated, the arguments of the proof of Theorem~4.8
in~\cite{BarneaErshovWeigel} (see Proposition~\ref{prop:BEW} below)
show that $L$ possesses a compact open normal subgroup. Since $L$
has no non-trivial compact normal subgroup, we deduce that $L$ is
discrete. Now $L$ is a non-trivial discrete normal subgroup and we
have already seen above how to finish the proof in that situation.

\smallskip

On the other hand, assume that the quasi-centre $\QZ(L)$ is trivial.
In particular $L$ possesses no non-trivial discrete normal subgroup
and we deduce from Proposition~\ref{prop:MinimalNormal} that the set
$\mathscr{M}$ of non-trivial minimal closed normal subgroups of $L$
is finite and non-empty. Since $L$ has no non-trivial compact normal
subgroup, no element of $\mathscr{M}$ is compact.

The group $G$ acts on $\mathscr{M}$ by conjugation. Let
$\mathscr{E}$ denote a $G$-orbit in $\mathscr{M}$. Since
$M_{\mathscr E} = \la M \; | \; M \in \mathscr{E}\ra$ is normal in
$G$, it is dense in $L$. By Proposition~\ref{prop:MinimalNormal}, we
infer that $\mathscr{E} =
\mathscr{M}$. In other words $G$
acts transitively on $\mathscr{M}$.

It now follows that $L$ is a quasi-product with the elements of
$\mathscr{M}$ as quasi-factors. In particular, any normal
subgroup $M'$ of any $M\in\mathscr M$ is normalised by $M$ and
centralised by each $N \in\mathscr M$ different from $M$. Since
$\prod_{N\in\mathscr M} N$ is dense in $L$, we infer that $M'$ is
normal in $L$. Since $M$ is a minimal normal subgroup of $L$, it
follows that $M$ is topologically simple and (i) holds.

\smallskip

Now we turn to the case where $G$ is not totally disconnected, hence
its identity component $G^\circ$ is cocompact. Since the maximal
compact normal subgroup  of Lemma~\ref{lem:LF:connected} is trivial,
$G^\circ$ is a connected Lie group.
Since its soluble radical is characteristic, it is trivial or cocompact.

In the former case, $G^\circ$ is semi-simple without compact factors.
Since its isotypic factors are
characteristic, there is only one isotypic factor and we conclude
that~(i) holds.

If on the other hand the radical of $G^\circ$ is cocompact, we
deduce that $G$ admits a characteristic cocompact connected soluble
subgroup $R\lhd G$. Let $T$ be the last non-trivial term of the
derived series of $R$. If the identity component $T^\circ$ is
trivial, then $T$ is a non-trivial discrete normal subgroup of $G$
and we may conclude by means of
Proposition~\ref{prop:discrete_normal}. Otherwise, the group $R$
possesses a characteristic connected Abelian subgroup $T^\circ$,
which is thus cocompact in $G$.

Since $T^\circ$ has no non-trivial compact subgroup, we have
$T^\circ\cong \RR^d$ for some $d$. The kernel of the homomorphism $G
\to \Out(T^\circ) = \Aut(T^\circ)$ is a cocompact normal subgroup
$N$ containing $T^\circ$ in its centre, and such that $N/T^\circ$ is
compact. In particular $N$ is a compactly generated locally compact
group in which every conjugacy class is relatively compact. In view
of U{\v{s}}akov's result (Theorem~\ref{thm:Ushakov}) and of the
triviality of $\LF(N)$, the group $N$ is of the form $\RR^n\times
\ZZ^m$. In conclusion, since $T^\circ$ is cocompact in $N$, we have
$T^\circ=N\cong \RR^n$. Considering once again the map $G \to
\Aut(T^\circ) \cong \mathbf{GL}_n(\RR)$, we deduce that $G/T^\circ$
is isomorphic to a compact subgroup of $\mathbf{GL}_n(\RR)$, which
has to be irreducible since otherwise $T^\circ$ would contain
a non-cocompact subgroup normalised by $G$. We conclude the proof
of Theorem~\ref{thm:monolithic} by recalling that every compact
subgroup of $\mathbf{GL}_n(\RR)$ is conjugated to a subgroup of
$\mathbf{O}(n)$.
\end{proof}

\section{Topological BN-pairs}\label{sec:BN}
By a celebrated lemma of Tits~\cite{Ti64}, any group admitting a
$BN$-pair of irreducible type has the property
that a normal subgroup acts either trivially or chamber-transitively on
the associated building. Tits used his
transitivity lemma to show in \emph{loc.~cit.} that if $G$ is perfect
and possesses a $BN$-pair with $B$
soluble, then any non-trivial normal subgroup is contained in $Z=
\bigcap_{g \in G} gBg\inv$. More generally, the
same conclusion holds provided $G$ is perfect and $B$ possesses a
soluble normal subgroup $U$ whose conjugates
generate $G$. If $G$ is endowed with a group topology, the same
arguments show that if $G$ is topologically
perfect and $U$ is pro-soluble, then $G/Z$ is topologically simple. The
following consequence of Theorem~\ref{thm:monolithic}
does not require any assumption on the normal subgroup structure of $B$.

\begin{cor}\label{cor:BN}
Let $G$ be a topological group endowed with a $BN$-pair of irreducible
type, such that $B<G$ is compact and
open. Then $G$ has a closed cocompact (topologically) characteristic
subgroup $H$ containing $Z = \bigcap_{g \in
G} gBg\inv$, such that $H/Z$ is topologically simple.
\end{cor}

It follows in particular that $G$ is \emph{topologically commensurable}
to a topologically simple group since
$Z$ is compact and $H$ cocompact.

\begin{proof}
The assumption that $B$ is compact open implies that $G$ is locally
compact and that the building $X$ associated with the given
$BN$-pair is locally finite. Since the kernel of this action
coincides with $Z= \bigcap_{g \in G} gBg\inv$, we may and shall
assume that $G$ acts faithfully on $X$. Since $G$ acts
chamber-transitively on $X$ and since $B$ is the stabiliser of some
base chamber $c_0$, it follows that $G$ is generated by the union of
$B$ with a finite set of elements mapping $c_0$ to each of its
neighbours. Thus $G$ is compactly generated. By Tits' transitivity
lemma~\cite[Prop.~2.5]{Ti64}, for any non-trivial normal subgroup
$N$ of $G$, we have $G= N.B$, whence $N$ is cocompact provided it is closed.

If $X$ is finite, then $G$ is compact and we are done. Otherwise
$G$ is non-compact and non-discrete, because $B$ is then necessarily
infinite. We are thus in a position to apply
Theorem~\ref{thm:monolithic}. Since $G$ is a subgroup of the totally
disconnected group $\Aut(X)$, it follows that $G$ is totally
disconnected and we deduce that $G$ is monolithic with a quasi-product
of topologically simple groups as a monolith. The fact that
the monolith has only one simple factor follows from the fact that
$G$ acts faithfully,  minimally and without fixed point at infinity
on a \cat realisation of $X$. Such a \cat realisation is necessarily
irreducible as a \cat space by~\cite{Caprace-Haglund} and
\cite[Theorem~1.10]{Caprace-Monod_structure} ensures that no abstract normal
subgroup of $G$ splits non-trivially as a direct product.
\end{proof}

\section{Discrete quotients}\label{sec:discrete}

\subsection*{Residually discrete groups}

Any topological group admits a filtering family of closed normal
subgroups, consisting of all \emph{open} normal subgroups.
Specialising Proposition~\ref{prop:FilterNormalSubgroups} to this
family yields the following fact.

\begin{cor}\label{cor:ResDiscrete}
Let $G$ be a compactly generated locally compact group. Then the
following assertions are equivalent.
\begin{enumerate}
\item $G$ is residually discrete.

\item $G$ is a totally disconnected SIN-group.

\item The compact open normal subgroups form a basis of identity
neighbourhoods.
\end{enumerate}
\end{cor}

Recall that a locally compact group is called a \textbf{SIN-group}\index{SIN-group} if it
possesses a basis of identity
neighbourhoods which are invariant by conjugation. Equivalently,
SIN-groups are those for which the left and right
uniform structures coincide. A classical theorem of Freudenthal and Weil
(\cite{Freudenthal36} and~\cite[\S\,32]{Weil40};
see also~\cite[\S\,16.4.6]{Dixmier6996})
states that a connected group is SIN if and only if it is of the form
$K\times \RR^n$ with $K$ compact (connected)
and $n \geq 0$.
This complements nicely the above corollary, implying easily that
\emph{any} locally compact SIN-group is an extension of a discrete group
by a group $K\times \RR^n$; the latter consequence is Theorem~2.13(1)
in~\cite{Grosser-Moskowitz71}.

\begin{proof}[Proof of Corollary~\ref{cor:ResDiscrete}]
The implications (iii)~$\Rightarrow$~(i) and (iii)~$\Rightarrow$~(ii)
are immediate.

\medskip \noindent%
(ii) $\Rightarrow$ (iii) Since $G$ is totally disconnected, the compact
open subgroups form a basis of identity
neighbourhoods~\cite[III \S\,4 No~6]{BourbakiTGI}. By assumption, given
any compact open subgroup $U < G$, there
is an identity neighbourhood $V \se U$ which is invariant by
conjugation. The subgroup generated by $V$ is
thus normal in $G$, open since it contains $V$ and compact since it is
contained in $U$. Thus (iii) holds
indeed.

\medskip \noindent%
(i) $\Rightarrow$ (iii) Assume that $G$ is residually discrete. Then $G$
is totally disconnected and the compact
open subgroups form a basis of identity neighbourhoods. Let $V < G$ be
compact and open and $Q_V \se V$
denote the compact normal subgroup provided by
Proposition~\ref{prop:FilterNormalSubgroups}. By assumption we
have $\bigcap \mathscr{F}=1$, where $\mathscr{F}$ denotes the collection
of all open normal subgroups of $G$.

We claim that the quotient $G/Q_V$ is residually discrete. Indeed, for
any $x \in G$, the family $\{(x\cdot N)
\cap Q_V\}_{N \in \mathscr{F}}$ is filtering and its intersection
coincides with $\{x\} \cap Q_V$. Since $Q_V$
is compact and since open subgroups are closed, it follows that for each
$x \not \in Q_V$ there exist finitely many elements $N_1, \dots, N_k \in
\mathscr{F}$ such
that $(\bigcap_{i=1}^k x\cdot N_i) \cap Q_V = \varnothing$. Since $N_0 =
\bigcap_{i=1}^k N_i$
belongs to $\mathscr{F}$, we have thus found an open normal subgroup
$N_0$ of $G$ such that $x \not \in N_0
\cdot Q_V$. The projection of $N_0 \cdot Q_V$ in the quotient $G/Q_V$ is
thus an open normal subgroup which
avoids the projection of $x$. This proves the claim.

Since the collection of all open normal subgroups of $G/Q_V$ forms a
filtering family,
Proposition~\ref{prop:FilterNormalSubgroups} now implies that $G/Q_V$
possesses some \emph{discrete} open normal
subgroup. Therefore $G/Q_V$ is itself discrete and hence $Q_V$ is open
in $G$. This shows that any compact open
subgroup $V$ contains a compact open normal subgroup $Q_V$. The desired
conclusion follows.
\end{proof}

\begin{remark}
Notice that a profinite extension of a discrete group is not necessarily
residually
discrete, as illustrated by the unrestricted wreath product $(\prod_\ZZ
\ZZ/2) \rtimes \ZZ$, where $\ZZ$ acts by
shifting indices. However, if a totally disconnected group $G$ possesses
a compact open normal subgroup $Q$ which is topologically finitely
generated, then $G$ is residually discrete. Indeed, the profinite group
$Q$ has then finitely many subgroups of
any given finite index and, hence, possesses a basis of identity
neighbourhoods consisting of
\emph{characteristic} subgroups.
\end{remark}

\subsection*{The discrete residual}

We recall that the \textbf{discrete residual}\index{discrete residual} of a topological group
is the intersection of all open normal subgroups. Notice that the
quotient of a group by its discrete residual is residually discrete.

\begin{proof}[Proof of Theorem~\ref{thm:OpenNormal}]
Let $R_0$ denote the discrete residual of $G$. Since $G/R_0$ is
residually discrete, it follows from
Corollary~\ref{cor:ResDiscrete} that $R_0$ is contained cocompactly in
some open normal subgroup of $G$. In view
of the hypotheses, this shows that $R_0$ is cocompact, whence compactly
generated.

Let $R_1$ denote the discrete residual of $R_0$. We have to show that
$R_0 = R_1$. Since $R_1$ is characteristic
in $R_0$, it is normal in $G$. Observe that all subquotients of $G/R_1$
considered below are totally disconnected
since the latter is an extension of the residually discrete groups
$R_0/R_1$ and $G/R_0$.
We consider the canonical projection $\pi : G \to G/R_1$
and define an intermediate characteristic subgroup $R_1 \leq L \leq R_0$ by
$$L = \pi\inv (\LF(R_0/R_1)).$$
Since the group $R_0/R_1$ is residually discrete, it follows from
Corollary~\ref{cor:ResDiscrete} that its
LF-radical is open. In other words the subquotient $R_0/L$ is discrete.
Since it is cocompact in $G/L$,
it is moreover finitely generated. It follows that the normal subgroup
$Z:=\centra_{G/L}(R_0/L)$ is open.
By hypothesis, it has finite index in $G/L$ and contains $R_0/L$. In
particular, it has cocompact centre
and thus $Z$ is compact-by-$\ZZ^m$ for some $m$, as is checked
\emph{e.g.} as an easy case of
Theorem~\ref{thm:Ushakov}, recalling that $Z$ is totally disconnected.
In conclusion, $Z$ has a compact open
characteristic subgroup; the latter has finite index in $G/L$ by
assumption. Thus $L$ is cocompact in $G$,
whence compactly generated. The topologically locally finite group $L/R_1$
is thus compact (Lemma~2.3 in~\cite{CapraceTD}). In particular $R_1$
itself is cocompact in $G$.
Now $G/R_1$ is a profinite group, thus residually discrete. This finally
implies that $R_0 = R_1$, as
desired.
\end{proof}

\subsection*{Quasi-discrete groups}

We end this section with an additional remark regarding the
quasi-centre. We shall say that a topological
group is \textbf{quasi-discrete}\index{quasi-discrete} if its quasi-centre is dense. Examples
of quasi-discrete groups include
discrete groups as well as profinite groups which are direct products of
finite groups. A connected group is
quasi-discrete if and only if it is Abelian. The following fact can be
extracted from the proof of Theorem~4.8
in~\cite{BarneaErshovWeigel}; since the argument is short we include it
for the sake of completeness.

\begin{prop}\label{prop:BEW}
In any quasi-discrete compactly generated totally disconnected
locally compact group, the compact open normal subgroups form a
basis of identity neighbourhoods.
\end{prop}

Thus a compactly generated totally disconnected locally compact
group which is quasi-discrete satisfies the equivalent conditions of
Corollary~\ref{cor:ResDiscrete}.

\begin{proof}
Let $G$ be as in the statement and $U<G$ be any compact open subgroup.
By compact generation,
there is a finite set $\{g_1, \dots, g_n\}$ that, together with $U$,
generates $G$.
Since $G$ is quasi-discrete, $G = \QZ(G) \cdot U$ and thus we can assume
that each $g_i$ belongs to
$\QZ(G)$. The subgroup $\bigcap_{i=1}^n \centra_U(g_i) < U$ is open and
hence contains a finite index open
subgroup $V$ which is normalised by $U$. Thus $V$ is a compact open
normal subgroup of $G$ contained in $U$.
\end{proof}

We finish this subsection by recording two consequences of the latter
for the sake of future reference.

\begin{cor}\label{cor:RmZn}
Let $G$ be a compactly generated locally compact group without
non-trivial compact normal subgroup. If $G$ is
quasi-discrete, then the identity component $G^\circ$ is  open, central
and isomorphic to $\RR^n$ for some $n$.
Moreover, we have $G = \QZ(G)$.
\end{cor}

\begin{proof}
We first observe that $G^\circ$ is central; indeed, it is centralised by
the dense subgroup $\QZ(G)$ since
$G^\circ$ is contained in every open subgroup. By
Lemma~\ref{lem:LF:connected}, the LF-radical of $G^\circ$ is
a compact normal subgroup of $G$, and is thus trivial by hypothesis;
moreover, it follows that $G^\circ$ is
an Abelian Lie group without periodic element. Thus $G^\circ \cong
\RR^n$ for some $n$.

Since any quotient of a quasi-discrete group remains quasi-discrete,
Proposition~\ref{prop:BEW} implies that the group of components
$G/G^\circ$ admits some compact open
normal subgroup $V$. It now suffices to prove $V=1$ to establish the
remaining statements.

Denote by $N\lhd G$ the $G^\circ$-by-$V$ extension; it is compactly
generated and has only compact conjugacy classes since $G^\circ$ is
central. In particular, Theorem~\ref{thm:Ushakov} guarantees that
$\LF(N)$ is compact and that $N/\LF(N)$ is Abelian without compact
subgroup. Since $N$ is normal in $G$, it follows that $\LF(N)$ is
trivial and thus indeed $N=G^\circ$ as required.
\end{proof}

\begin{cor}
Let $G$ be a compactly generated totally disconnected locally compact
group admitting an open quasi-discrete subgroup. Then either
$G$ is compact or $G$ possesses an infinite discrete quotient.
\end{cor}

\begin{proof}
Let $H$ be the given quasi-discrete open subgroup of $G$. Let $h \in
\QZ(H)$ be an element of the quasi-centre
of $H$. Then $\centra_H(h)$ is open in $H$, from which we infer that
$\centra_G(h)$ is open in $G$ and hence $h
\in \QZ(G)$. In particular,  the closure $Z = \overline{\QZ(G)}$ is open
in $G$.

If $Z$ has infinite index, then $G/Z$ is  an infinite discrete quotient
of $G$. Otherwise $Z$
has finite index and is thus compactly generated. By
Proposition~\ref{prop:BEW}, it follows that $Z$ possesses a compact open
normal subgroup. Since $Z$ has finite index in $G$, we
deduce that $G$ itself possesses a compact open normal subgroup. The
desired conclusion follows.
\end{proof}

\section{On structure theory}\label{sec:structure}

\subsection*{Quasi-simple quotients}

Before undertaking the proof of Theorems~\ref{thm:structure:first}
and~\ref{thm:structure:lower}, we record one additional consequence
of Proposition~\ref{prop:FilterNormalSubgroups}. Before stating it,
we recall from the introduction that a group is called
\textbf{quasi-simple}\index{quasi-simple} if it possesses a cocompact normal subgroup
which is topologically simple and contained in every non-trivial
closed normal subgroup of $G$; in other words, a quasi-simple group
is a monolithic group whose monolith is cocompact and topologically
simple.

\begin{cor}\label{cor:SimpleQuotients}
Let $G$ be a compactly generated locally compact group and $\{N_v \;
| \; v \in \Sigma\}$ be a collection of pairwise distinct closed
normal subgroups of $G$ such that for each $v \in \Sigma$, the
quotient $G/N_v$ is quasi-simple, non-discrete and non-compact.
Suppose that $\bigcap_{v \in \Sigma} N_v = 1$. Then $\Sigma$ is
finite.
\end{cor}

\begin{proof}
We write $H_v:=G/N_v$. By hypothesis each $H_v$ is monolithic with
simple cocompact monolith, which we denote by $S_v$.
We claim that $G$ has no non-trivial compact normal subgroup. Let indeed
$Q \lhd G$ be a compact normal subgroup
of $G$. By the assumptions made on $H_v$, the image of $Q$ in $H_v$ is
trivial for each $v \in \Sigma$.
Thus $Q \se \bigcap_{v \in \Sigma} N_v$ and hence $Q$ is trivial.

The same line of argument shows that $G$ has no non-trivial soluble
normal subgroup. In particular, the identity component $G^\circ$ is
a connected semi-simple Lie group with trivial centre and no compact
factor, see Lemma~\ref{lem:LF:connected}. Such a Lie group $G^\circ$
is the direct product of its simple factors. Moreover, $G$ has an
open characteristic subgroup of finite index which splits as a
direct product of the form $G^\circ \times D$ determining some
compactly generated totally disconnected group $D$, see \emph{e.g.}
(the proof of) Theorem~11.3.4 in~\cite{Monod_LN}.

The identity component of each $H_v$ coincides with the image of
$G^\circ$ (since any quotient of a totally
disconnected group is totally disconnected). Thus, whenever $H_v$ is not
totally disconnected, the hypothesis
implies $S_v=H_v^\circ$ and $H_v$ is a profinite extension of one of the
simple factors of $G^\circ$, and that
each factor appears once.

At this point, we can and shall assume that $G$ is totally
disconnected.

\smallskip

In view of Proposition~\ref{prop:discrete_normal} and the assumption
made on $H_v$, the group $S_v$ is non-discrete for each $v \in
\Sigma$. In particular, it follows that $H_v$ has trivial
quasi-centre.

Since the image of $\QZ(G)$ in $H_v$ is contained $\QZ(H_v)$, we
deduce that $\QZ(G)N_v = N_v$ for all $v \in \Sigma$. In other
words, we have $\QZ(G) < \bigcap_{v \in \Sigma} N_v = 1$ and we
conclude that $G$ has trivial quasi-centre.

Let now $\mathscr{F}$ be the filter of closed normal subgroups
of $G$ generated by $\{N_v \; | \; v \in \Sigma\}$. Since $G$ has no
compact non-trivial normal subgroup and no non-trivial discrete normal
subgroup (as $\QZ(G) =1$), we deduce from
Proposition~\ref{prop:FilterNormalSubgroups} that $\mathscr{F}$ is
finite. Thus $\Sigma$ is finite as well, as desired.
\end{proof}

\subsection*{Maximal normal subgroups}

We shall need the following statement due to R.~Grigorchuk and
G.~Willis; since it is unpublished, we provide a proof for the
reader's convenience.

\begin{prop}\label{prop:GriWi}
Let $G$ be a totally disconnected compactly generated non-compact
locally compact group. Then $G$ admits a non-compact quotient with
every proper quotient compact.
\end{prop}

\begin{proof}
By Zorn's lemma, it suffices to prove that for any chain $\mathscr
H$ of non-cocompact closed normal subgroups $H\lhd G$, the group
$M=\overline{\bigcup_{H\in \mathscr H} H}$ is still non-cocompact.
If not, then $M$ is compactly generated. Therefore, choosing a
compact open subgroup $U<G$, the chain $\{H.U\}$ of groups is an
open covering of $M$, whence there is $H\in \mathscr H$ with
$H.U\supseteq M$. Then this $H$ is cocompact, which is absurd.
\end{proof}

\begin{remark} \label{rem:GriWi}
In view of the structure theory of connected
groups~\cite{Montgomery-Zippin}, the above Proposition holds also
true in the non-totally-disconnected case.
\end{remark}

The following is a dual companion to
Proposition~\ref{prop:MinimalNormal}. Additional information in this
direction will be provided in Proposition~\ref{prop:MinMax} in
Appendix~\ref{app:B} below.

\begin{prop}\label{prop:MaximalNormal}
Let $G$ be a compactly generated totally disconnected locally
compact group which possesses no infinite discrete quotient, and let
$H = \Res(G)$ be the discrete residual of $G$. Then every proper
closed normal subgroup of $H$ is contained in a maximal one, and
the set $\mathscr N$ of proper maximal closed normal subgroups is
finite.
\end{prop}

\begin{proof}
By Theorem~\ref{thm:OpenNormal}, the discrete residual $H$ is
cocompact in $G$, hence compactly generated. Furthermore it has no
non-trivial finite quotient. Since $H$ is totally disconnected, any
compact quotient would be profinite, and we infer that $H$ has no
non-trivial compact quotient. Now the same argument as in the proof
of Proposition~\ref{prop:GriWi} using Zorn's lemma shows that every
proper  closed normal subgroup of $H$ is contained in a maximal one.

Let $\mathscr N$ denote the collection of all these. Any quotient
$H/N$ being topologically simple, hence quasi-simple, the finiteness
of $\mathscr N$ follows readily from Corollary~\ref{cor:SimpleQuotients}.
\end{proof}

\subsection*{Upper and lower structure}

\begin{proof}[Proof of Theorem~\ref{thm:structure:first}]
We assume throughout that $G$ is non-compact since otherwise~(ii)
holds trivially. Assume first that $G$ is almost connected. In
particular the neutral component $G^\circ$ coincides with the
discrete residual of $G$. Let $R$ denote the maximal connected
soluble normal subgroup of $G$; this \textbf{soluble radical}\index{soluble radical}\index{radical!soluble}
is indeed well defined even if $G$ is not a Lie group as proved by K.~Iwasawa
(Theorem~15 in~\cite{Iwasawa49}; see also~\cite[(3.7)]{Paterson}. If
$R$ is cocompact we are in case (ii) of the Theorem. Otherwise $G$
is not amenable and using the structure theory of connected groups
(notably Theorem~4.6 in~\cite{Montgomery-Zippin}), we deduce that
$G^\circ/R$ possesses a non-compact (Lie\mbox{-})simple factor, so
that all assertions of the case~(iii) of the Theorem are satisfied.

We now assume that $G$ is not almost connected. If $G$ admits an
infinite discrete quotient we are in case~(i) of the Theorem. We
assume henceforth that $G$ has no infinite discrete quotient. In
particular its discrete residual $G^+$ is cocompact and admits
neither non-trivial discrete quotients  nor disconnected compact
quotients, see Corollary~\ref{cor:MaxCptQuotient}. Moreover $G^+$ is
compactly generated, non-compact and contains the identity component
$G^\circ$.

Let $\mathcal S$ be the collection of all topologically simple
quotients of $G^+$. Applying Corollary~\ref{cor:SimpleQuotients} to
the quotient group $G^+/K$, where $K = \bigcap_{S \in \mathcal S}
\Ker(G^+ \to S)$, we deduce that $\mathcal S$ is finite. Thus the
assertion~(iii) of Theorem~\ref{thm:structure:first} will be
established provided we show that $\mathcal S$ is non-empty.

To this end, it suffices to prove that the group of components
$G^+/G^\circ$ admits some non-compact topologically simple quotient.
But this follows from Proposition~\ref{prop:MaximalNormal} since any
quotient of $G^+/G^\circ$ is non-compact and since each
topologically simple quotient is afforded by a maximal closed normal
subgroup.
\end{proof}

\begin{proof}[Proof of Theorem~\ref{thm:structure:lower}]
We assume that assertions~(i) and~(ii) of the Theorem fail. Note
that if $G$ is totally disconnected, then
Proposition~\ref{prop:MinimalNormal} finishes the proof.

As is well known (see the proof of
Corollary~\ref{cor:SimpleQuotients}), the non-existence of
non-trivial compact (resp. connected soluble) normal subgroups
implies that $G$ possesses a characteristic open subgroup  of finite
index $G^+<G$ which splits as a direct product of the form $G^+
 = G^\circ \times D$, where $G^\circ$ is a semi-simple Lie group and
$D$ is totally disconnected.

Notice that $G/G^\circ$ is a totally disconnected locally compact group
which might possess non-trivial finite normal subgroups. In order to
remedy this situation, we shall now exhibit a closed normal subgroup
$G_1 \leq G$ containing $G^\circ$ as a finite index subgroup and such
that $G/G_1$ has no non-trivial compact or discrete normal subgroup.

Let $N$ be a closed normal subgroup of $G$ containing $G^\circ$.
Then $N^+ = N \cap G^+$ is a finite index subgroup which decomposes
as a direct product of the form $N^+ \cong G^\circ \times (D \cap
N)$. If the image of $N$ in $G/G^\circ$ is compact (resp. discrete),
then $D \cap N$ is a compact (resp. discrete) normal subgroup of
$G$, and must therefore be trivial, since otherwise assertion (ii)
(resp.~(i)) would hold true. We deduce that any compact (resp.
discrete) normal subgroup of $G/G^\circ$ is finite with order
bounded above by $[G:G^+]$. In particular, there is a maximal such
normal subgroup, and we denote by $G_1$ its pre-image in $G$.

Since $G_1 \cap D$ injects into $G_1/G^\circ$, it is a finite normal
subgroup of $G$ and must therefore be trivial. Moreover $G^+ =
G^\circ D$ is closed in $G$. Thus  $D$ has closed image in
$G/G^\circ$, whence in $G/G_1$ since the canonical projection
$G/G^\circ \to G/G_1$ is proper, as it has finite kernel. This
implies that $D G_1$ is closed in $G$. In other words $\la D \cup
G_1 \ra$ is a  characteristic closed subgroup of finite index in $G$
which is isomorphic to $G_1 \times D$.

It follows at once that there is a canonical one-to-one
correspondence between the closed normal subgroups of $G$ contained
in $D$ and the closed normal subgroups of $G/G_1$ contained in
$DG_1/G_1$.

Now $G/G_1$ is a compactly generated totally disconnected locally
compact group without non-trivial compact or discrete normal
subgroup, and Proposition~\ref{prop:MinimalNormal} guarantees that
the set $\mathscr M_1$ of its non-trivial minimal closed normal
subgroups is finite and non-empty. Moreover, an element of $\mathscr
M_1$ does not possess any non-trivial finite index closed normal
subgroup and must therefore be contained in $DG_1/G_1$. Similarly,
any minimal closed normal subgroup of $G$ must be contained in
$G^+$.

Since any minimal closed normal subgroup of $G$ is either connected
or totally disconnected, and since the connected ones are nothing
but (regrouping of) simple factors of $G^\circ$, we finally obtain a
canonical one-to-one correspondence between $\mathscr M_1$ and the
set of non-trivial minimal closed normal subgroups of $G$ which are
totally disconnected. The desired conclusion follows since, as
observed above, the set $\mathscr M_1$ is finite and non-empty.
\end{proof}

\begin{proof}[Proof of Corollary~\ref{cor:CharSimple}]
Assume $G$ is not discrete. The discrete residual $\Res(G)$ is
characteristic. If $\Res(G)= 1$ then $G$ is residually discrete and
hence, its LF-radical is open by Corollary~\ref{cor:ResDiscrete}.
Since the LF-radical is characteristic and $G$ is not discrete, we
deduce that $G$ is topologically locally finite, hence compact since
it is compactly generated.

We assume henceforth that $\Res(G) = G$ and that $G$ is not compact.
The above argument shows moreover that $G$ has trivial LF-radical
and trivial quasi-centre.

If $G$ is not totally disconnected, then it is connected. If this is the
case, the LF-radical of $G$ is compact
(see Lemma~\ref{lem:LF:connected}) hence trivial, and we deduce that $G$
is a Lie group. In this case, the
standard structure theory of connected Lie groups allows one to show
that either $G \cong \RR^n$ or $G$ is a
direct product of pairwise isomorphic simple Lie groups.

Assume finally that $G$ is totally disconnected. Then
Proposition~\ref{prop:MinimalNormal} guarantees that the set
$\mathscr M$ of non-trivial minimal closed normal subgroups of $G$
is finite and non-empty. Moreover, since for any proper subset
$\mathscr E \subset \mathscr M$, the subgroup $\overline{\la M \; |
\; M \in \mathscr E \ra}$ is properly contained in $G$, it follows
that $\Aut(G)$ acts transitively on $\mathscr M$.

Now we conclude as in the proof of Theorem~\ref{thm:monolithic} that
$G$ is a quasi-product with the elements of $\mathscr M$ as
topologically simple quasi-factors.
\end{proof}

\section{Composition series with topologically simple subquotients}

We start with an elementary decomposition result on quasi-products.

\begin{lem}\label{lem:DecomposeQuasiProduct}
Let $G$ be a locally compact group which is a quasi-product with
infinite topologically simple quasi-factors $M_1, \dots, M_n$.
Then $G$ admits a sequence of closed normal subgroups
$$1 = Z_0 < G_1 < Z_1 < G_2 < Z_2 < \cdots < Z_{n-1} < G_n = G,$$
where for each $i=1, \dots, n$, the subgroup $G_i $ is defined as
$G_i= \overline{ Z_{i-1} M_i}$, the subquotient  $G_i /Z_{i-1}$ is
topologically simple and $Z_i/G_i = \centra(G/G_i)$.
\end{lem}

\begin{proof}
The requested properties of the normal series provide in fact a
recursive definition for the closed normal subgroups $Z_i$ and $G_i$.
In particular all we need to show is that $Z_{i-1}$ is a maximal proper
closed normal subgroup
of $G_i$.

We first claim that $M_j \cap G_{i-1} =1$ for all $1 \leq i < j$,  where
it is understood that $G_0 = 1$. Since no $M_j$ is Abelian, this amounts
to showing that $G_{i-1} \leq \centra_G(M_i M_{i+1} \dots M_n)$. We proceed
by induction on $i$, the base case $i=1$ being trivial. Now we need to
show that $M_i$ and $Z_{i-1}$ are both contained in $ \centra_G(M_{i+1}
\dots M_n)$.  This is clear for $M_i$. By induction $M_i \dots M_n$ maps
onto a dense normal subgroup of $G/G_{i-1}$. Since $Z_{i-1}/G_{i-1} =
\centra(G/G_{i-1})$ by definition, we infer that $[Z_{i-1}, M_j] \leq
G_{i-1}$ for all $j \geq i$. Of course we have also $ [Z_{i-1}, M_j]
\leq M_j$ since $M_j$ is normal. The intersection $G_{i-1} \cap M_j$
being trivial by induction, we infer that $[Z_{i-1}, M_j]$ is trivial as
well. Thus $M_i$ and $Z_{i-1}$ are indeed both contained in $
\centra_G(M_{i+1} \dots M_n)$, and so is thus $G_i$. The claim stands
proven.

Let now $N$ be a closed normal
subgroup $G$ such that $Z_{i-1} \leq N < G_i$. By the claim we have
$G_i \cap M_j = 1$ for all $j >i$, hence $N \cap M_j = 1$. Now if $N
\cap M_i \neq 1$, then $M_i < N$ since $M_i$ is topologically
simple. In that case, we deduce that $N$ contains $ Z_{i-1} M_i$,
which contradicts that $N$ is properly contained in $G_i$. Thus we
have $N \cap M_i = 1$. In particular we deduce that  $N \leq \centra_G(M_i
M_{i+1} \dots M_n)$. Since $M_i M_{i+1} \dots M_n$ maps densely into
$G/G_{i-1}$, we deduce that the image of $N$ in $G/G_{i-1}$ is
central. By definition, this means that $N$ is contained in $Z_{i-1}$,
thereby proving that $Z_{i-1}$ is indeed maximal normal in $G_i$.
\end{proof}

\begin{proof}[Proof of Theorem~\ref{thm:CompositionSeries}]
In view of the structure theory of connected locally compact groups
(see Lemma~\ref{lem:LF:connected}) and of connected Lie groups, the
desired result holds in the connected case. Moreover, any
homomorphic image of a Noetherian group is itself Noetherian.
Therefore, there is no loss of generality in replacing $G$ by the
group of components $G/G^\circ$. Equivalently, we shall assume
henceforth that $G$ is totally disconnected.

We first claim that any closed normal subgroup of $G$ is compactly
generated. Indeed, given such a subgroup $N < G$, pick any compact
open subgroup $U$ and consider the open subgroup $NU < G$. Since $N$
is a cocompact subgroup of $NU$, which is compactly generated as $G$
is Noetherian, we infer that $N$ itself is compactly generated, as
claimed. Notice that the same property is shared by closed normal
subgroups of any open subgroup of $G$.

Let now $\Res(G)$ denote the discrete residual of $G$. Thus
$G/\Res(G)$ is residually discrete and Noetherian.
Corollary~\ref{cor:ResDiscrete} thus implies that the LF-radical of
$G/\Res(G)$ is open, while the above claim guarantees that it is
compact. We denote by $O$ the pre-image in $G$ of $\LF(G/\Res(G))$.
Thus $O$ is an open characteristic subgroup of $G$ containing
$\Res(G)$ as a cocompact subgroup. In particular, the discrete
quotients of $O$ are all finite. Now Theorem~\ref{thm:OpenNormal}
guarantees that $\Res(G) = \Res(O)$ has no non-trivial discrete
quotient.

\smallskip%
Setting $H = \Res(G)$, we have thus far constructed a series $1 < H
< O < G$ of characteristic subgroups with $O$ open and $O/H$
compact. We shall now construct inductively a finite increasing
sequence
$$1 = H_0 < H_1 < H_2 < \cdots < H_l = H<O$$
of normal subgroups of $O$ satisfying the following conditions for
all $i = 1, \dots, l$:
\begin{itemize}
\item[(a)] If $\LF(H/H_{i-1})$ is non-trivial, then $\LF(G/H_i) = 1$.

\item[(b)]
$H_i/H_{i-1}$ is either compact, or isomorphic to
$\ZZ^n$ for some $n$, or to a quasi-product with topologically
simple pairwise $O$-conjugate quasi-factors.
\end{itemize}

Let $j > 0$ and assume that the first $j-1$ terms $H_0, \dots, H_{j-1}$
of the desired series have already been
constructed, in such a way that properties (a)  and (b) hold with $i<j$.
We proceed to define $H_j$ as
follows.

\smallskip

If $\LF(H/H_{j-1})$ is non-trivial, then we let $H_j$ be the
pre-image in $H$ of $\LF(H/H_{j-1})$. Properties (a) and (b) clearly
hold for $j=i$ in this case.

\smallskip

Assume now that $\LF(H/H_{j-1}) = 1$ and  that $\QZ(H/H_{j-1})$ is
non-trivial.  Let $M$ denote the closure of $\QZ(H/H_{j-1})$ in
$H/H_{j-1}$. Thus $M$ is characteristic and quasi-discrete.
Furthermore, the fact that $\LF(H/H_{j-1}) = 1$ implies that $M$ has
no non-trivial compact normal subgroup. Therefore
Corollary~\ref{cor:RmZn} ensures that $M = \QZ(H/H_{j-1})$ and that
the identity component $M^\circ$ is open and isomorphic to $\RR^n$
for some $n$.

Since $H$ is totally disconnected, it follows that $M^\circ$ is
trivial. Thus $M$ is totally disconnected as well, hence
compact-by-discrete in view of Proposition~\ref{prop:BEW}. But $M$
has no non-trivial compact normal subgroup since  $\LF(H/H_{j-1})$
is trivial and it follows that $M$ is discrete. We claim that $M$ is
Abelian. Indeed, since $M$ is discrete and finitely generated, its
centraliser in $H/H_{j-1}$ is open. By assumption $H$ has no
non-trivial discrete quotient, and this property is inherited by the
quotient $ H/H_{j-1}$. We deduce that $\centra_{H/H_{j-1}}(M) =
H/H_{j-1}$; in other words $M$ is central in $H/H_{j-1}$ hence
Abelian, as claimed. Let $M_0$ denote the unique maximal free
Abelian subgroup of $M$. Then $M_0$ is non-trivial since $M$ is not
compact. We define $H_j$ to be the pre-image of $M_0$ in $H$. Then
$H_j$ is characteristic and again, properties (a) and (b) are both
satisfied with $i=j$ in this case.

\smallskip

It remains to define $H_j$ in the case where the LF-radical
$\LF(H/H_{j-1})$ and the quasi-centre $\QZ(H/H_{j-1})$ are both
trivial. In that case, Proposition~\ref{prop:MinimalNormal}
guarantees that $H/H_{j-1}$ contains some non-trivial minimal closed
normal subgroups of $O/H_{j-1}$, say $M$, provided $H/H_{j-1}$ is
non-trivial. Clearly $M$ is characteristically simple, so that, by
Corollary~\ref{cor:CharSimple}, it is a quasi-product with finitely
many topologically simple quasi-factors. Now $O$ acts
transitively by conjugation on these quasi-factors, otherwise $M$
would contain a proper closed normal subgroup (see
Proposition~\ref{prop:MinimalNormal}), contradicting minimality.
It remains to define $H_j$ as the pre-image of $M$ in $O$.

Hence (a) and (b) hold with $i=j$ in all cases.

\smallskip
We have thus constructed an ascending chain of subgroups $1=H_0 <
H_1 < H_2  < \cdots < H< O$ which are all normal in $O$ and we
proceed to show that $H_k = H$ for some large enough index $k$.
Suppose for a contradiction that this is not the case and set
$H_\infty = \overline{\bigcup_{i=1}^\infty H_i}$. Since $H_\infty$
is normal in $O$, it is compactly generated (see the second
paragraph of the present proof above). Let $V < H_\infty$ be a
compact open subgroup. Then the ascending chain $V\cdot H_1 < V\cdot
H_2 < \cdots$ yields a covering of $H_\infty$ by open subgroups. The
compact generation of $H_\infty$ thus implies that $V\cdot H_k =
H_\infty$ for $k$ large enough. In particular $H_k$ is cocompact in
$H_\infty$. Therefore $\LF(H/H_{k})$ is non-trivial. By property
(a), this implies that $\LF(H/H_{k+1})$ is trivial and hence
$H_\infty \se H_{k+1}$. This contradiction establishes the
claim.

\smallskip%
It only remains to show that the series of characteristic subgroups
$1=H_0 < H_1 < H_2  < \cdots < H_l = H$ that we have constructed can
be refined into a subnormal series satisfying the desired conditions
on the subquotients. By construction, it suffices to refine the
non-compact non-Abelian subquotients $H_i/H_{i-1}$. Since these are
quasi-products with finitely many topologically simple pairwise
$O$-conjugate quasi-factors, we may replace $O$ by an
appropriate closed normal subgroup of finite index, say $O'$, in
such a way that for all $i$, each topologically simple quasi-factor
of $H_i/H_{i-1}$ is normal in $O'/H_{i-1}$. Consider now the decomposition
of $H_i/H_{i-1}$ provided by Lemma~\ref{lem:DecomposeQuasiProduct}.
Each term of this decomposition is normal in $O'/H_{i-1}$, and must
therefore be compactly generated. Therefore, the corresponding
subquotients are compactly generated. In particular, the Abelian
subquotients are compact-by-$\ZZ^k$. Introducing these intermediate
terms in the series $1=H_0 < H_1 < H_2  < \cdots < H_l = H < O' < O
< G$, we obtain a refinement which has all the desired properties.
\end{proof}


\appendix
\renewcommand{\thesection}{\Roman{section}}

\section{The adjoint closure and asymptotically central sequences}

\subsection*{On the Braconnier topology}

Let $G$ be a locally compact group and $\Aut(G)$ denote the group of all
homeomorphic automorphisms of $G$.
There is a natural topology, sometimes called the \textbf{Braconnier
topology}\index{Braconnier topology}, turning $\Aut(G)$ into a Hausdorff topological group;
it is defined by the sub-base of identity neighbourhoods
$$\mathfrak A(K,U) := \big\{\alpha \in \Aut(G) \; | \;  \forall x \in K,
\ \alpha(x) x^{-1} \in U \text{ and } \alpha^{-1}(x) x^{-1} \in U \big\},$$
where $K \se G$ is compact and $U \se G$ is an identity neighbourhood
(see Chap.~IV \S\,1 in~\cite{Braconnier}
or~\cite[Theorem~26.5]{HewittRoss}).

In other words, this topology is the common refinement of the
compact-open topology for automorphisms and their inverses; recall in
addition that a
topological group has canonical uniform structures so that the
compact-open topology coincides with the topology of uniform convergence
on compact sets
(\cite{Braconnier} p.~59 or~\cite{KelleyGTM} \S\,7.11).

\medskip

In fact, the Braconnier topology coincides with the restriction of the
$g$-topology on the group of all homeomorphisms of $G$ introduced by
Arens~\cite{Arens46}, itself hailing from Birkhoff's
$C$-convergence~\cite[\S\,11]{Birkhoff34}. It can alternatively be
defined by restricting
the compact-open topology for the Alexandroff compactification, an idea
originating with van Dantzig and
van der Waerden~\cite[\S\,6]{vanDantzig-vanderWaerden}.

\medskip

Braconnier shows by an example that the compact-open topology itself is
in general too coarse to turn $\Aut(G)$ into a topological
group~\cite[pp.~57--58]{Braconnier}. We shall establish below a basic
dispensation from this fact for the adjoint representation
(Proposition~\ref{prop:CptOpen}). Meanwhile, we recall
that the Braconnier topology coincides with the compact-open topology
when $G$ is compact (Lemma~1 in~\cite{Arens46}) and when $G$ is locally
connected
(Theorem~4 in~\cite{Arens46}). There are of course non-locally-connected
connected groups: the solenoids of Vietoris~\cite[II]{Vietoris27} and
van Dantzig~\cite[\S\,2 Satz~1]{vanDantzig30}. Nevertheless, using
notably the solution to Hilbert's fifth problem, S.P.~Wang showed
that the two topologies still coincide for all connected and indeed
almost connected locally compact groups~\cite[Corollary~4.2]{Wang69}.
Finally, the topologies coincide for $G$ discrete and
$G=\mathbf{Q}_p^n$, see~\cite[p.~58]{Braconnier}.

\medskip

We emphasise that the Braconnier topology on $\Aut(G)$ need not be
locally compact, see~\cite[\S\,26.18.k]{HewittRoss}. A criterion
ensuring that $\Aut(G)$ is locally compact will be presented in
Theorem~\ref{thm:AutG:lc} below in the case of totally
disconnected groups.

Nevertheless, $\Aut(G)$ is a Polish (hence Baire) group when $G$ is
second countable. Indeed,
it is by definition closed (even for the weaker pointwise topology) in
the group of homeomorphisms of $G$ endowed with Arens'
$g$-topology; the latter is second countable (see
\emph{e.g.}~\cite[5.4]{Gleason-Palais}) and complete for the
\emph{bilateral}
uniform structure~\cite[Theorem~6]{Arens46}. Notice that this complete
uniformisation is not the usual left or right uniform structure,
which is known to be sometimes incomplete at least for the group of
homeomorphisms (Arens, \emph{loc.\ cit.}).

The Baire property implies for instance that $\Aut(G)$ is discrete when
countable, which was observed in~\cite[Satz~2]{Plaumann}
for $G$ itself discrete.

\subsection*{Adjoint representation}

Given a closed normal subgroup $N < G$, the conjugation action of
$G$ on $N$ yields a map $G\to \Aut(N)$ which is continuous (see
\cite[Theorem~26.7]{HewittRoss}). In particular, the natural map $\Ad
: G \to \Aut(G)$ induced by the conjugation action is a continuous
homomorphism. We endow the group $\Ad(G)<\Aut(G)$ with the Braconnier
topology.
Thus, a sub-base of identity neighbourhoods is given by the image in
$\Ad(G)$ of all subsets
of $G$ of the form
$$\mathfrak B(K,U) := \big\{g \in G \; | \;  [g,K]  \se U \text{ and
}[g^{-1}, K] \se U\big\},$$
where $(K,U)$ runs over all pairs of compact subsets and identity
neighbourhoods of $G$.

\medskip

As an abstract group, $\Ad(G)$ is isomorphic to $G/\centra(G)$; we
emphasise however that the latter
is endowed with the generally finer quotient topology.

\begin{prop}\label{prop:CptOpen}
Let $G$ be a locally compact group such that the group of components
$G/G^\circ$ is unimodular.

Then the Braconnier topology on $\Ad(G)$ coincides with the compact-open
topology.
\end{prop}

\begin{proof}
Let $\{g_\alpha\}_\alpha$ be a net in $G$ such that $\Ad (g_\alpha)$
converges to the identity in the compact-open topology.
According to the result of S.P.~Wang quoted earlier in this section, the
automorphisms $\Ad(g_\alpha)|_{G^\circ}$ of the identity component
$G^\circ$ converge to the identity for the Braconnier topology on
$\Aut(G^\circ)$. According to Proposition~2.3 in~\cite{Wang69},
it now suffices to prove that the induced automorphisms on $G/G^\circ$
also converge to the identity for the Braconnier topology
on $\Aut(G/G^\circ)$. Therefore, we can suppose henceforth that $G$ is
totally disconnected.

By assumption, $\{g_\alpha\}$ eventually penetrates every set of the form
$$\mathfrak B'(K',U') := \{g \in G \; | \;  [g,K']  \se U'\},$$
where $K' \se G$ is compact and $U' \se G$ is a neighbourhood of $e\in
G$. Thus it suffices to show that for all
$K\se G$ compact and $U\se G$ identity neighbourhood, there is $K'$ and
$U'$ with
$$\mathfrak B'(K',U')\ \se\ \mathfrak B(K,U).$$
Since $G$ is totally disconnected, there is a compact open subgroup
$U'<G$ contained in $U$. Set $K'=K\cup U'$ and fix any
$g\in \mathfrak B'(K',U')$. We need to show that $[g\inv, K] \se U$.

First, notice that $[g\inv, K] = g\inv [g, K]\inv g$. Next, $[g, K]$ and
hence also $[g, K]\inv$ is in $U'$. Finally, $[g, U']\se U'$
means that $g U' g\inv\se U'$; by unimodularity, it follows that $g$
normalises $U'$. We conclude that $[g\inv, K] \se U'\se U$, as was to be
shown.
\end{proof}

A locally compact group $G$ for which the map $\Ad : G \to \Inn(G)$
is closed will be called \textbf{Ad-closed}\index{Ad-closed}. In that case, $\Inn(G)$ is
isomorphic to
$G/\centra(G)$ as a topological group and thus in particular it is
locally compact.

The group $G$ can fail to be \Adc even when it is a connected Lie group
(Example~\ref{ex:Lie:nonadc} below; see also
\emph{e.g.}~\cite{Lee-Wu,Zerling76}).
Perhaps more strikingly, $G$ can fail to be \Adc even when countable,
discrete and $\mathbf{Z}/2\mathbf{Z}$-by-abelian~\cite[4.5]{Wu71}.

\subsection*{Asymptotically central sequences}

Let $G$ be a locally compact group. A sequence $\{g_n\}$ of elements of
$G$ is called \textbf{asymptotically central}\index{asymptotically central sequence} if $\Ad(g_n)$ converges to
the identity in $\Ad(G)$. Obvious examples are central sequences or
sequences converging to~$e$; we shall investigate the existence of
non-obvious ones (for an admittedly limited analogy, compare the
\emph{property~$\Gamma$} introduced for $\mathrm{II}_1$-factors by
Murray and von Neumann, Definition~6.1.1 in~\cite{Murray-vonNeumannIV}).

\smallskip

The existence of suitably non-trivial asymptotically central sequences
is related to the question whether the Braconnier topology on $G$
(strictly speaking, on $G/\centra(G)$) coincides with the initial
topology, as follows.

\begin{prop}\label{prop:Adc}
Let $G$ be a second countable locally compact group. The following
conditions are equivalent.
\begin{enumerate}[(i)]
\item $\Inn(G)$ is locally compact.

\item The continuous homomorphism  $\Ad: G \to \Inn(G)$ is closed.

\item The map $G/\centra(G)\to \Ad(G)$ is a topological group isomorphism.

\item The image in $G/\centra(G)$ of every asymptotically central
sequence is relatively compact.
\end{enumerate}

\noindent
A sufficient condition for this is that $G$ admits some compact open
subgroup $U$ such that $\norma_G(U)$ is compact.
\end{prop}

\begin{proof}
(i) $\Rightarrow$ (ii) This is a well-known application of the Baire
category principle, going back at least
to~\cite[Theorem~XIII]{Pontrjagin39}.

\medskip
\noindent
(ii) $\Rightarrow$ (iii) and (iii) $\Rightarrow$ (iv) follow from the
definitions.

\medskip
\noindent
(iv) $\Rightarrow$ (i) Let $\{K_n\}$ be an increasing sequence of
compact subsets of $G$ whose union covers $G$ and let $\{U_n\}$ be a
decreasing family of sets providing a basis of neighbourhoods of $e\in
G$. Assuming for a contradiction that $\Inn(G)$ is not locally compact,
none of the sets $\mathfrak B(K_n, U_n)$ can have a relatively compact
image in $\Inn(G)$. Therefore, we can choose for each $n$ an element
$g_n$ in  $\mathfrak B(K_n, U_n)$ but not in $K_n.\centra(G)$. By
construction, the sequence $\{g_n\}$ is unbounded in $G/\centra(G)$ but
$\Ad(g_n)$ converges to the identity, a contradiction.

\medskip
Finally, notice that if $U$ is a compact open subgroup of $G$, then
$\norma_G(U) = \mathfrak B(U, U)$. This shows that if $\norma_G(U)$ is
compact, then $\Inn(G)$ admits $\Ad(\mathfrak B(U, U))$ as a compact
identity neighbourhood.
\end{proof}

We recall from~\cite{KakutaniKodaira} that a $\sigma$-compact locally
compact group $G$ always possesses a compact normal subgroup $Q$ such
that the quotient $G/Q$ is metrisable. In particular, any compactly
generated locally compact group without non-trivial compact normal
subgroup satisfies the hypotheses of Proposition~\ref{prop:Adc}.

The following construction provides examples of Lie groups  which
are not \Adc.

\begin{example}\label{ex:Lie:nonadc}
Let $L \cong \RR < \RR^2$ be a one-parameter subgroup with
irrational slope and denote by $Z$ the image of $L$ in the torus
$\mathbf{T}^2 = \RR^2/\ZZ^2$. Thus $Z$ is a connected dense
subgroup of $\mathbf{T}^2$. Let us now choose a continuous faithful
representation of $\mathbf{T}^2$ in $O(4)$ and consider the
corresponding semi-direct product $H = \mathbf{T}^2 \ltimes \RR^4$.
We define $G = Z \ltimes \RR^4$. Thus $G$ is a connected subgroup of
the Lie group $O(4) \ltimes \RR^4$.

We claim that $G$ is not \Adc. Indeed, let $(z_n)$ be an unbounded
sequence of elements of $Z$ which converge to $1$ in the torus
$\mathbf{T}^2$. One verifies  easily that $G$ is centrefree and
that the above sequence is asymptotically central in $G$. This
yields the desired claim in view of Proposition~\ref{prop:Adc}.
\end{example}

An illustration of the relevance of the notion of \Adc groups is
provided by the following.

\begin{lem}\label{lem:centraliser:Adc}
Let $G$ be a locally compact group and $H<G$ be a closed subgroup. If
$H$ is \Adc, then $H.\centra_G(H)$ is closed in $G$.
\end{lem}
\begin{proof}
Without loss of generality, we may assume that $H.\centra_G(H)$ is
dense in $G$. Then $H$ is normal and there is a continuous
conjugation action $\alpha: G \to \Aut(H)$. Since $H.\centra_G(H)$
is dense, it follows that $\Inn(H)$ is dense in $\alpha(G)$. Now
$\Inn(H)$ being closed in $\Aut(H)$ by hypothesis, we infer that
$\alpha(G) = \Inn(H)$. The result follows, since the pre-image of
$\Inn(H)$ in $G$ is nothing but $H.\centra_G(H)$.
\end{proof}

\subsection*{The adjoint closure}

The closure of $\Inn(G)$ in $\Aut(G)$ will be called the \textbf{adjoint
closure}\index{adjoint closure} of $G$ and will be denoted by $\overline{\Inn(G)}$.
We think of an automorphism in $\overline{\Inn(G)}$ as ``approximately inner''.
We point out that $\Inn(G)$ is normal in $\Aut(G)$ and hence in
particular in $\overline{\Inn(G)}$.

\smallskip
Basic properties of the adjoint closure are summarised in the following.
Notice that $\overline{\Inn(G)}$ is not assumed locally compact except
in the last item.

\begin{lem}\label{lem:AdjointClosure}
Let $G$ be a locally compact group.
\begin{enumerate}[(i)]
\item If $G$ is centrefree, then so is $\overline{\Inn(G)}$.

\item If $G$ is topologically simple, then so is $\overline{\Inn(G)}$.

\item If $G$ is totally disconnected, then so is $\overline{\Inn(G)}$.

\item Suppose that $\overline{\Inn(G)}$ is locally compact. If $G$ is
compactly generated, then so is $\overline{\Inn(G)}$.
\end{enumerate}
\end{lem}

\begin{proof}
(i) Given $\alpha \in \centra_{\Aut(G)}(\Inn(G))$, we have $\alpha(g)x
\alpha(g)\inv = g x g\inv$ for all $g, x \in G$. Thus $\alpha(g)\inv g $
belongs to $\centra(G)$ and the result follows.

\medskip
\noindent
(ii) Let $H < \overline{\Inn(G)}$ be a closed normal subgroup and
let $H_0 = \Ad\inv(H)$ be the pre-image of $H$ in $G$. Then $H_0$ is
a closed normal subgroup of $G$ and is thus trivial of the whole group. If $H_0 = G$, then
$H$ contains $\Inn(G)$ which is dense, thus $H=G$ as well.
If $H_0 = 1$, then $H \cap \Inn(G) = 1$. This implies that
$[H, \Inn(G)] \se H \cap \Inn(G) = 1$. Thus $H $ commutes
with the dense subgroup $\Inn(G)$ and is thus contained in the
centre of $\overline{\Inn(G)}$, which is trivial by the assertion~(i).

\medskip
\noindent
(iii) See~\cite[IV \S\,2]{Braconnier} or~\cite[Theorem~26.8]{HewittRoss}.

\medskip
\noindent
(iv) Let $U$ be a compact neighbourhood of the identity in
$\overline{\Inn(G)}$ and $C\se G$ a compact
generating set. Then $U.\Ad(C)$ generates $\overline{\Inn(G)}$.
\end{proof}

\subsection*{Locally finitely generated groups}

We shall say that a totally disconnected locally compact group $G$ is
\textbf{locally finitely generated}\index{locally finitely generated} if $G$ admits some
compact open subgroup that is topologically finitely generated,
\emph{i.e.} possesses a finitely generated dense subgroup.
Since any two compact open subgroups of $G$ are commensurable, it follows
that $G$ is locally finitely generated if and only if \emph{any}
compact open subgroup is topologically finitely generated.
Examples of such include $p$-adic analytic groups (see
\emph{e.g.}~\cite[Theorem~8.36]{AnalyticPro-p}), many complete
Kac--Moody groups over finite fields~\cite[Theorem~6.4]{CER} as well as
several (but not all) locally compact groups acting properly on locally
finite trees~\cite{Mozes98}.

An important property of finitely generated profinite groups is that
they admit a (countable) basis of identity neighbourhoods consisting of
characteristic subgroups, because they have only finitely many closed
subgroups of any given index. Locally finitely generated groups are thus
covered by the following result.

\begin{thm}\label{thm:AutG:lc}
Let $G$ be a totally disconnected compactly generated locally compact group.
Suppose that $G$ admits an open subgroup $U$ that has a basis of identity
neighbourhoods consisting of characteristic subgroups of $U$ (\emph{e.g.} $G$ is
locally finitely generated).

Then $\Aut(G)$ is locally compact.
\end{thm}

The proof will use the following version of the Arzel\`a--Ascoli Theorem.

\begin{prop}\label{prop:AA}
Let $G$ be a locally compact group and $V \se \Aut(G)$ a subset such that
\begin{enumerate}[(i)]
\item $V=V\inv$;

\item $G$ has arbitrarily small $V$-invariant identity neighbourhoods;

\item $V(x)$ is relatively compact in $G$ for each $x\in G$.
\end{enumerate}
\noindent
Then $V$ is relatively compact in $\Aut(G)$.
\end{prop}

In the case where $V$ is a compact subgroup of $\Aut(G)$, this is
Theorem~4.1 in~\cite{Grosser-Moskowitz67}.

\begin{proof}[Proof of Proposition~\ref{prop:AA}]
Point~(ii) implies that $V$ is equicontinuous (in fact, uniformly
equicontinuous). Therefore, we can apply Arzel\`a--Ascoli
(in the generality of~\cite{BourbakiTGII}, X~\S\,2 No.~5) and deduce
that $V$ has compact closure in the space of continuous maps $G\to G$
endowed with the compact-open topology (which, as mentioned, coincides
with the topology of compact convergence).
The closure of $V$ remains in the space of continuous endomorphisms
since the latter is closed even pointwise.
In view of the symmetry of the assumptions and of the fact that
composition is continuous in the compact-open
topology~\cite[XII.2.2]{Dugundji66}, the closure of $V$ remains in
$\Aut(G)$ and is compact for the Braconnier topology.
\end{proof}

\begin{proof}[Proof of Theorem~\ref{thm:AutG:lc}]
Let $U<G$ be an open subgroup admitting a basis of identity
neighbourhoods $\{U_\alpha\}_{\alpha}$ consisting of characteristic
subgroups of $U$. We can assume $U$ compact upon intersecting with a compact open subgroup.
Let $C\se G$ be a symmetric compact set generating $G$ and containing $U$.
We shall prove that $V :=\mathfrak A(C,U)\se \Aut(G)$ satisfies the assumptions of
Proposition~\ref{prop:AA}; this then establishes the theorem.

The first assumption holds by definition. For the second, notice first
that $V$ normalises $U$ since $U\se C$ implies
$$V\ \se\ \mathfrak A(U,U)\ =\ \norma_{\Aut(G)}(U).$$
Assumption~(ii) holds since the identity neighbourhoods $U_\alpha$ are
characteristic, hence normalised by $ \norma_{\Aut(G)}(U)$.

In order to establish the last assumption, choose $x\in G$. Since $C$ is
generating and symmetric, there is an integer $d$ such that
$x\in C^d$. The definition of $V$ shows that for any automorphism $n\in
V$, we have $n(x)\in (U.C)^d$; this implies~(iii).
\end{proof}

\subsection*{The adjoint closure of discrete groups}
A particularly simple illustration of the concepts introduced above is provided by discrete groups.
The Braconnier topology on $\Aut(G)$ is then the topology of pointwise convergence and coincides with
pointwise convergence of the inverse. The adjoint closure $\overline{\Inn(G)}$ coincides therefore with
the group $\Linn(G)$ of \textbf{locally inner}\index{locally inner automorphism} automorphisms,
\emph{i.e.} automorphisms that coincide
on every finite set with some inner automorphism. This concept was apparently first introduced
({\cyrrm lokal\cprime no vnutrennim}) by Gol'berg~\cite[\S\,3 {\cyrrm Opredelenie~5}]{Golberg46}.

\smallskip

Here are a few elementary properties of the resulting correspondance $G\mapsto\overline{\Inn(G)}$ from
abstract (resp. countable) groups to topological (resp. Polish) groups.

\begin{prop}\label{prop:Discrete:1}
Let $G$ be a discrete group and $A=\overline{\Inn(G)}=\Linn(G)$ its adjoint closure.

\begin{enumerate}
\item $\Ad(G) \leq \QZ(A)$. In particuylar $A$ is quasi-discrete; it is discrete if and only if there is a finite set $F\se G$
with $\centra_G(F)=\centra(G)$.

\item $A$ is compact if and only if $G$ is
FC\footnote{Recall that $G$ is an \textbf{FC-group}\index{FC-group} if all its conjugacy classes are finite.}.

\item $A$ is locally compact if and only if  there is a finite set $F\se G$ such that
the index $[\centra_G(F):\centra_G(F')]$ is finite for every finite $F'\supseteq F$ in $G$.

\item $A$ is locally compact and compactly generated if and only if there is $F$ as in~(iii) and
$F_0\se G$ finite such that $F_0\cup \centra_G(F)$ generates $G$.
\end{enumerate}
\end{prop}

\begin{proof}
All verifications are straightforward. One uses notably an elementary version of Proposition~\ref{prop:AA}
stating that, for $G$ discrete, a subset $V\se \Aut{G}$ has compact closure
if and only if $V(x)$ is finite for all $x\in G$.
\end{proof}

Discrete groups are a safe playground to experiment with intermediate topologies inbetween the original
topology and the Braconnier topology induced \emph{via} the adjoint representation. The following construction
will lead to interesting examples, see Appendix~\ref{app:B} and especially Example~\ref{ex:1bis}.

\smallskip

Let $N$ be a discrete group and $U<N$ a subgroup such that
\begin{enumerate}
\item $\centra_U(g)$ has finite index in $U$ for all $g\in N$;

\item the intersection of all $N$-conjugates of $U$ is trivial.
\end{enumerate}

In particular, the $N$-conjugates of $U$ generate a completable group
topology~\cite[TG\,III, \S\,3, No~4]{BourbakiTGIII} and $N$ injects
into the resulting complete totally disconnected topological group $M$.

\begin{prop}\label{prop:Discrete:2}
\ 
\begin{enumerate}
\item The group $M$ is locally compact; in fact $U$ has compact-open closure in $M$.

\item There is a (necessarily unique) continuous injective homomorphism $M/\centra(M)\to \overline{\Inn(N)}$
compatible with the maps $N\to M$ and $N\to \overline{\Inn(N)}$.
In particular, the dense image of $N$ in $M$ is normal and quasi-central (thus $M$ is quasi-discrete).

\item If $N$ is centrefree (resp. simple), then $M/\centra(M)$ is centrefree (resp. topologically simple).

\item $M/\centra(M)=\overline{\Inn(N)}$ if and only if $\centra_N(F)\se U$ for some finite $F\se N$.
\end{enumerate}
\end{prop}

\begin{proof}
The first assertion is due to the fact that the closure of $U$ in $M$ is a quotient of the profinite
completion of $U$. The second assertion follows from the fact that a system of neighbourhoods of the identity
for the $M$-topology on $N$ is given by $U\cap V$, where $V$ ranges over the $\overline{\Inn(N)}$-neighbourhoods
of the identity. The third assertion follows by the same argument as in the proof of Lemma~\ref{lem:AdjointClosure}. For the last assertion, observe that $M=\overline{\Inn(N)}$ if and only if $U$ is open in
the Braconnier topology on $N$.
\end{proof}

\begin{example}\label{ex:1} 
Let $\Omega$ be a countably infinite set, let $N < \Sym(\Omega)$ be an infinite (almost) simple group of alternating \textbf{finitary permutations}\index{finitary permutation}, \emph{i.e.} permutations with finite support. Choose also  an equivalence relation $\sim$ on $\Omega$ all of whose equivalence classes have finite cardinality, and let $U< N$ be an infinite subgroup preserving each equivalence class. For each $g \in N$, there is a finite index subgroup $U' < U$ such that $g$ and $U'$ have disjoint support. Therefore $\centra_U(g)$ has finite index in $U$ for all $g \in N$. Moreover the intersection of all $N$-conjugates of $U$ is trivial since $N$ is almost simple.

Concretely, one could define $N$ as the group of all alternating finitary permutation, which is indeed simple, and define  $U$ as the subgroup preserving all equivalence classes of a relation $\sim$ which is a partition into subsets of fixed size $k>1$. The group $U$ is then isomorphic to a restricted direct sum of finite alternating groups of degree $k$ and $M$ is a  totally disconnected locally compact group which is topologically simple, topologically locally finite and quasi-discrete. (Here $M$ is centrefree because every asymptotically central sequence of $N$ converges pointwise to the identity.)

We remark that the examples of topologically simple locally compact groups admitting a dense normal subgroup which were constructed by G.~Willis in~\cite[\S\,3]{Willis07} all fit in this set-up, and can all be obtained by taking various specializations of the groups $N < \Sym(\Omega)$ and $U < N$. 
\end{example}

\begin{remark}
The previous example takes advantage of the fact that the group of all finitary permuations of a countably infinite set $\Omega$ is not \Adc. Notice however that its adjoint closure, which incidentally coincides with the group $\Sym(\Omega)$ of \emph{all} permutations of $\Omega$, is however not locally compact. Proposition~\ref{prop:Discrete:2} and Example~\ref{ex:1} thus correspond to completions which are genuinely intermediate between $\Ad(N)$ and $\overline{\Ad(N)}$. This is an instance of a general scheme that we shall present below, see Proposition~\ref{prop:QuasiDirect:constrution}.
\end{remark}

\section{Quasi-products and dense normal subgroups}\label{app:B}

For general locally compact groups, there is a naturally occurring
structure that is weaker than direct products. We establish its basic
properties and give some examples. In order to avoid some obvious
degeneracies, it is good to have in mind the centrefree case.

\subsection*{Definitions and the Galois connection}

Let $G$ be a topological group. We call a closed normal subgroup $N\lhd
G$ a \textbf{quasi-factor}\index{quasi-factor!in a group} of $G$ if $N.\centra_G(N)$ is dense in $G$.
In other words, this means that the $G$-action on $N$ is ``approximately inner'' in
the sense that the image of $G$ in $\Aut(N)$ is contained in the adjoint
closure $\overline{\Inn(N)}$.

\medskip

If $N$ is a quasi-factor, then $N\cap\centra_G(N)$ is contained in the
centre of $G$. Thus, in the centrefree case, quasi-factors provide an
example of the following concept with $p=2$:

We say that $G$ is the \textbf{quasi-product}\index{quasi-product} of the closed normal
subgroups $N_1, \ldots, N_p$ if the multiplication map
$$N_1\times\cdots\times N_p \lra G$$
is injective with dense image. We call the groups $N_i$ the
\textbf{quasi-factors}\index{quasi-factor!of a quasi-product} of this quasi-product; notice that $N_i$ and
$N_j$ commute for all $i\neq j$ and therefore each $N_i$ is indeed a
quasi-factor in the earlier sense.

\medskip

Given a quasi-product, one has a family of quotients $G\tto S_i$ defined
by $S_i=G/\centra_G(N_i)$. Notice that the image of $N_i$ in $S_i$ is a
dense normal subgroup; moreover, when $G$ is centrefree, $N_i$ injects
into $S_i$. Therefore, we obtain an injection with dense image:
$$G/\centra(G)\ \lra\ S_1\times\cdots\times S_p.$$
(The relation between quasi-products and dense normal subgroups will be
further investigated below; see Example~\ref{ex:1}.)

\medskip

The map $N\mapsto \centra_G(N)$ is an antitone Galois connection on the
set of closed normal subgroups of $G$ and in particular also on the
collection of quasi-factors. It turns out that this correspondence
behaves particularly well for certain groups appearing in the main
results of this article, as follows. Denote by $\Max$ (resp. $\Min$) the
set of all maximal (resp. minimal) closed normal subgroups which are
non-trivial.

\begin{prop}\label{prop:MinMax}
Let $G$ be a non-trivial compactly generated totally disconnected
locally compact group. Assume that $G$ is centrefree and without
non-trivial discrete quotient. If $\bigcap \Max$ is trivial, then the
following hold.
\begin{enumerate}[(i)]
\item $\Min$ and $\Max$ are finite and non-empty.

\item The assignment $N \mapsto \centra_G(N)$ defines a bijective
correspondence from $\Min$ to $\Max$.

\item Every element of $\Min \cup \Max$ is a quasi-factor.

\item $G$ is the quasi-product of its minimal normal subgroups.
\end{enumerate}
\end{prop}

This result provides in particular additional information on
characteristically simple groups, which supplements
Corollary~\ref{cor:CharSimple}. Indeed, the hypotheses of the
proposition are in particular fulfilled by characteristically simple
groups falling in case (iv) of Corollary~\ref{cor:CharSimple} (see also
Proposition~\ref{prop:MaximalNormal}).

\begin{proof}[Proof of Proposition~\ref{prop:MinMax}]
Notice that $G$ has no non-trivial compact quotient since every discrete
quotient is trivial. The set $\Max$ is non-empty since it has trivial
intersection; its finiteness follows from
Theorem~\ref{thm:structure:first}. Moreover, $G$ embeds in the product
of simple groups $\prod_{K \in \Max} G/K$, which implies that $G$ has
trivial quasi-centre and no non-trivial compact normal subgroup. Indeed,
otherwise some $G/K$ would be compact or quasi-discrete (as in the proof
of Corollary~\ref{cor:SimpleQuotients}). In particular,
Theorem~\ref{thm:structure:lower} implies that $\Min$ is finite and
non-empty; assertion~(i) is established. Actually, we shall use below
not only that $\Min$ is non-empty, but that every non-trivial closed
normal subgroup of $G$ contains a minimal one, see
Proposition~\ref{prop:MinimalNormal}.

\medskip

For the duration of the proof, denote by $\Maxqf$ the subset of $\Max$
consisting of those elements which are quasi-factors of $G$.

\medskip \noindent
\emph{We claim that  the map $N \mapsto \centra_G(N)$ defines a
one-to-one correspondence of  $\Min$ onto $\Maxqf$. Moreover, every
element of $\Min \cup \Maxqf$ is a quasi-factor. }

\medskip
Let $N \in \Min$. By hypothesis there is some $K \in \Max$ which does
not contain $N$. By minimality of $N$ we deduce that $N \cap K$ is
trivial and hence $[N,K] = 1$. Therefore $N.K$ is dense in $G$ by
maximality of $K$. In particular, $N$ and $K$ are both quasi-factors.
Moreover, since $\centra_G(N)$ contains $K$, maximality implies
$K=\centra_G(N)$ because $G$ is centrefree. In other words, $N \mapsto
\centra_G(N)$ defines a map $\Min\to\Maxqf$. Since any minimal closed
normal subgroup of $G$ different from $N$ commutes with $N$, it is
contained in $K$. Therefore, the above map is an injection of  $\Min$
into $\Maxqf$. It remains to show that it is surjective.

To this end, pick $K \in \Maxqf$. Then $\centra_G(K)$ is a non-trivial
closed normal subgroup of $G$. It therefore contains an element of
$\Min$, say $N$. By definition $K$ is contained in $\centra_G(N)$,
whence $K = \centra_G(N)$ by maximality. The claim stands proven.

\medskip \noindent
\emph{We claim that every element of $\Max \setminus \Maxqf$ contains
every element of $\Min$.}

\medskip
If $K \in \Max$ and $N \in \Min$ are such that $N \not \leq K$, then $N$
and $K$ commute and, hence, $K = \centra_G(N)$. Thus $K \in \Maxqf$ by
the previous claim.

\medskip \noindent
\emph{We claim that $\bigcap_{K \in \Maxqf} K =1$.}

\medskip

Otherwise $\bigcap_{K \in \Maxqf} K$ would contain some $N \in \Min$,
which is also contained in every $L \in \Max \setminus \Maxqf$ by the
previous claim. This contradicts the hypothesis that $\bigcap \Max$ is
trivial.

\medskip \noindent
\emph{We claim that $G = \overline{[G,G]}$.}

\medskip
By hypothesis $G$ has no non-trivial discrete quotient. This property is
clearly inherited by any quotient of $G$. Therefore, the claim follows
from the fact that the only totally disconnected locally compact Abelian
group with that property is the trivial group.

\medskip \noindent
\emph{We claim that $G = \overline{\la N \; | \; N \in \Min\ra}$.}

\medskip
Set $H = \overline{\la N \; | \; N \in \Min\ra}$ and $A = G/H$. It
follows from the first claim above that every element $K \in \Maxqf$ has
dense image in $A$. In view of the third claim above, we infer that  $A$
admits dense normal subgroups $L_1, \dots, L_p$ with trivial
intersection. In view of Lemma~\ref{lem:nilp} below, it follows that $A$
is nilpotent, hence trivial by  the previous claim.

\medskip \noindent
\emph{We claim that $\Max = \Maxqf$.}

\medskip
Indeed, every element of $\Max \setminus \Maxqf$ contains every element
of $\Min$. By the previous claim, this implies that every element of
$\Max \setminus \Maxqf$ coincides with $G$ and thus fails to be a
non-trivial subgroup.

\medskip Now assertions (ii), (iii) and (iv) follow at once.
\end{proof}

\begin{lem}\label{lem:nilp}
Let $A$ be a Hausdorff topological group containing $p$ dense normal
subgroups $L_1, \dots, L_p$ such that $\bigcap_{i=1}^p L_i = 1$. Then
$A$ is nilpotent of degree~$\leq p-1$
\end{lem}
\begin{proof}
For each $j = 1, \dots, p$, we set $M_j = \bigcap_{i=j}^p L_i $. In
particular $M_1$ is trivial and $M_p = L_p$.

Set $A_i =  A/ \overline{M_i}$ for all $i = 1, \dots, p$. We have a
chain of continuous surjective maps
$$A \cong A_1 \to A_2 \to \dots \to A_{p-1} \to A_p \cong 1.$$
Let $i <p$. Since  $M_i = L_i \cap M_{i+1}$, it follows that the
respective images of $L_i$  and $M_{i+1}$ in $A_i$ commute. Since
moreover $L_i$ is dense in $A$, it maps densely in $A_i$ and we deduce
that the image of $M_{i+1}$ in $A_i$ is central. In particular $A_i$ is
a central extension of $A_{i+1}$. It readily follows that the upper
central series of $A$ terminates after at most $p-1$ steps.
\end{proof}

\subsection*{On the non-Hausdorff quotients of a quasi-product}
The following result describes the \emph{algebraic} structure of the
generally non-Hausdorff quotient $G/N_1\cdots N_p$ (its topological
structure being trivial). It applies in particular to the case of
totally disconnected groups that are
Noetherian.

\begin{prop}\label{prop:Abelian}
Let $G$ be a totally disconnected locally compact group that is a
quasi-product with quasi-factors $N_1, \dots, N_p$.

\begin{enumerate}[(i)]
\item If $N_i$ possesses a maximal compact subgroup $U_i$ for some $i \in
\{1, \dots, p\}$, then for each compact  subgroup $U < G$ containing
$U_i$, the
quotient  $U/U_i.(U \cap \centra_G(N_i))$ is Abelian.

\item  If $N_i$ possesses a maximal compact subgroup $U_i$ for each $i \in
\{1, \dots, p\}$, then the quotient $G/\centra(G).N_1 \cdots N_p$ is
Abelian.
\end{enumerate}
\end{prop}

\begin{proof}
Let $U_i< N_i$ be a maximal compact  subgroup and let $U<G$ be a
compact open subgroup containing $U_i$. Since $U \cap N_i$ is a
compact subgroup of $N_i$ containing $U_i$, we have $U \cap N_i =
U_i$ by maximality. Let also $Z_i = \centra_G(N_i)$. We shall first
show that $U / U_i.(U \cap Z_i)$ is Abelian, which is the
assertion~(i).

In order to establish this, consider the open subgroup $H_i :=
U.N_i$. Since $U_i$ is a maximal compact subgroup of $N_i$, it follows
that $U$ is a maximal compact subgroup of $H_i$.

Notice that $H_i \cap Z_i$ is centralised by a cocompact
subgroup of $H_i$, namely $N_i$. Therefore $H_i \cap Z_i$ is a
topologically $\overline{\text{FC}}$-group,
indeed the $H_i$-conjugacy class of every element is relatively compact.
By U{\v{s}}akov's result~\cite{Ushakov} (see Theorem~\ref{thm:Ushakov} above),
the set $Q$ of all its periodic elements coincides with the LF-radical and
the quotient $(H_i \cap Z_i)/Q$ is torsion-free Abelian (and discrete by total disconnectedness).
Since $H_i \cap Z_i$
is normal in $H_i$, it follows that $Q$ is normalised by $U$. Thus we can form the
subgroup $U \cdot Q$ of $H_i$, which is topologically locally finite and hence compact since $U$ was maximal compact in $H_i$.
This implies $Q \leq U$ and in particular we have $Q \leq U \cap Z_i=: V_i$.

Since $V_i $ is normalised by $U$ and
centralised by $N_i$, it is a compact normal subgroup of $H_i$ contained
in $H_i \cap Z_i$. By definition, this implies that $V_i$ is contained
in $Q$. Thus $Q = V_i$ and the quotient $(H_i \cap Z_i)/V_i$ is thus
torsion-free Abelian.

Notice that $Z_i$ contains $N_j$ for all $j \neq i$. Therefore
$N_i.Z_i$ is dense in $G$ and, since $H_i$ is open, it follows that
$H_i \cap (N_i.Z_i)= N_i. (H_i \cap Z_i)$ is dense in $H_i$.
Therefore the Abelian group $(H_i \cap Z_i)/V_i$ maps densely to
$H_i/N_i . V_i \cong U / U \cap (N_i.V_i) = U/(U \cap N_i).(U \cap
Z_i)$. We deduce that
the latter is Abelian, as claimed.

\medskip
Our next claim is that $G/N_i.Z_i$ is Abelian. Indeed, we have $G =
U.N_i.Z_i$ since $N_i.Z_i$ is dense. We deduce that $G/N_i. Z_i
\cong U/ U \cap (N_i. Z_i)$, which may be viewed as a quotient of
the group $U/(U \cap N_i).(U \cap Z_i)$. The latter is known to be
Abelian by (i), which confirms the present claim.

\medskip
Suppose now that each $N_i$ contains some maximal compact subgroup
$U_i$. The above discussion shows that the derived group $[G, G]$ is
contained in the intersection $N := \bigcap_{i=1}^p N_i.Z_i$. In
other words the quotient $G/N$ is Abelian, and it only remains to
show that
$$N = N_1 \cdots N_p . \centra(G).$$

Let $g \in N_1.Z_1 \cap N_2.Z_2$ and write $g = n_1z_1=n_2 z_2$ with
$n_i \in N_i$ and $z_i \in Z_i$. Then $n_1\inv z_2 =  z_1 n_2\inv $
belongs to $Z_1 \cap Z_2$ since $N_i \se Z_j$ for all $i \neq
j$. Since $Z_1 \cap Z_2 = \centra_G(N_1.N_2)$, we deduce that $g \in
N_1.N_2.\centra_G(N_1.N_2)$. This shows that $N_1.Z_1 \cap N_2. Z_2
\se N_1.N_2.\centra_G(N_1.N_2)$. Since the opposite inclusion
obviously holds true, we have in fact  $N_1.Z_1 \cap N_2. Z_2 =
N_1.N_2.\centra_G(N_1.N_2)$. A straightforward induction now shows
that
$$\bigcap_{i=1}^p N_i.Z_i = N_1\cdots N_p.\centra_G(N_1\cdots N_p).$$
Since $N_1 \cdots N_p$ is dense in $G$, we have
$\centra_G(N_1\cdots N_p) = \centra(G)$, from which the
assertion~(ii) follows.
\end{proof}

\subsection*{Quasi-products with Ad-closed quasi-factors}

The following gives a simple criterion for a quasi-product to be
direct.

\begin{lem}\label{lem:QuasiDirect:Adc}
Let $G$ be a locally compact group that is a quasi-product with
quasi-factors $N_1, \dots, N_p$.

If $N_1, \dots, N_{p-1}$ are \Adc and centrefree, then $G \cong N_1
\times \cdots \times N_p$.
\end{lem}

\begin{proof}
We work by induction on $p$, starting by noticing that the statement
is empty for $p=1$. Since $[N_1, N_i] \se N_1 \cap N_i = 1$ for
$i >1$, we deduce that $N_2\cdots  N_p \se \centra_G(N_1)$. In
particular $N_1.\centra_G(N_1)$ is dense in $G$. From
Lemma~\ref{lem:centraliser:Adc} and the fact that $N_1$ has trivial
centre, it follows that $G \cong N_1 \times \centra_G(N_1)$. By
projecting $G$ onto $ \centra_G(N_1)$, we deduce that the product
$N_2\cdots N_p $ is dense in $ \centra_G(N_1)$. Thus $
\centra_G(N_1)$ is the quasi-product of $N_2, \dots, N_p$. The
desired result follows by induction.
\end{proof}

It will be shown in the next subsection that, conversely, a group
$N$ which is not \Adc may often be used to construct a non-trivial
quasi-product having $N$ as a quasi-factor, see
Example~\ref{ex:2} below.

\subsection*{Non-direct quasi-products and dense analytic normal subgroups}

We propose a general scheme to construct quasi-products out of a pair of topological groups
$M,N$ together with a faithful continuous $M$-action by automorphisms on $N$. The intuition is
that $M$ plays the role of \emph{some} adjoint completion of $N$ appearing in a quasi-direct
product with two quasi-factors isomorphic to $N$. The precise set-up is as follows.

\medskip

Let $M, N$ be topological groups and $\alpha:M\hookrightarrow \Aut(N)$ an injective continuous representation.
In complete generality, \emph{continuity} shall mean that the map $M\times N\to N$ is jointly continuous;
therefore, when considering locally compact groups, it suffices to assume that $\alpha$ is a
continuous homomorphism for the Braconnier topology on $\Aut(N)$.

In order to formalise the idea that $M$ is a generalisation of the Ad-closure $\overline{\Inn(N)}$,
we assume throughout
$$\Inn(N)\se\alpha(M) \hspace{.7cm} \text{ and }\hspace{.7cm}  \overline{\alpha\inv(\Inn(N))} = M.$$
Thus $\alpha(M)$ is indeed intermediate in $\Inn(N)\se\alpha(M)\se\overline{\Inn(N)}$. The trivial case,
\emph{i.e.} direct product, of our construction will be characterised by $\alpha(M)=\Inn(N)$. On the other
hand, already $\alpha(M)= \overline{\Inn(N)}$ will produce interesting examples.

\smallskip

We denote by $\alpha_\Delta:M\hookrightarrow \Aut(N\times N)$ the diagonal action, which is still
injective and continuous. We form the semi-direct product
$$H:=\ (N\times N) \rtimes_{\alpha_\Delta} M,$$
which is a topological group for the multiplication
$$(n_1, n_2, m)(n'_1, n'_2, m') = (n_1 \alpha(m)(n'_1), n_2 \alpha(m)(n'_2), m m').$$
We observe that the set
$$Z:=\ \Big\{ (n, n, m) : \alpha(m)=\Inn(n)\inv \Big\}$$
is a subgroup of $H$ and we write $G:=H/Z$. For convenience, we write $N_1=N\times 1$
and $N_2=1\times N$, which we view as subgroups of $H$.

\begin{prop}\label{prop:QuasiDirect:constrution}
~
\begin{itemize}
\item[\it (i)] $Z$ is a closed normal subgroup of $H$ (thus we consider $G$ as a topological group). 

\item[\it (ii)]  The morphism $N_i\to G$ is a topological isomorphism onto its image,
which is closed and normal in $G$; we thus identify $N_i$ and its image.
The resulting quotients $G/N_i$ are topologically isomorphic to $M$.

\item[\it (iii)] The morphism $N\times N\to G$ has dense image; the latter
is properly contained in $G$ if and only if $\alpha(M)\neq \Inn(N)$.
The kernel is the diagonal copy of $\centra(N)$ (in particular, if $N$ is centre-free,
$G$ is a quasi-product).

\item[\it (iv)]  $\centra_G(N_i) = N_{3-i}$; in particular, if $N$ is centre-free, so is $G$.

\item[\it (v)]  If $N$ is topologically simple, then $G$ is characteristically
simple. Moreover, $G$ cannot be written non-trivially as a direct product
unless $\alpha(M)= \Inn(N)$.
\end{itemize}
\end{prop}

\begin{proof}
(i) The fact that $Z$ is closed follows from the fact that
the diagonal in $N\times N$ is closed and that $\alpha$ is continuous.
A computation shows that $N\times N$ centralises $Z$, whilst $M$ normalises it;
hence $Z$ is normal.

\medskip \noindent (ii) The morphism $N_1\to G$ is continuous and injective.
Suppose that a net $(n_\beta, 1, 1)\in N_1$ converges to some $(n, n', m)$ modulo
$Z$. Then there are nets $\nu_\beta\in N$, $\mu_\beta\in M$ with
$\alpha(\mu_\beta)=\Inn(\nu_\beta)\inv$ such that $(n_\beta \nu_\beta, \nu_\beta, \mu_\beta)$
tends to $(n, n', m)$. Thus $n_\beta$ converges in $N$ (to $n {n'}\inv$) and hence
the morphism is indeed closed. The image is normal by definition and the case of $N_2$ is analogous.

It is straightforward to show that $H = N_i.Z.M$ and that $N_i.Z\cap M =1$. Therefore
$H/N_i.Z \cong M$, as claimed.

\medskip \noindent (iii)
The density is equivalent to the density of
$Z.(N \times N)$ in $H$, which follows from the conditions on $\alpha$.
The additional statement on the image follows from the canonical identification of coset sets
$$G/(N\times N)\ \cong\ M/\alpha\inv(\Inn(N)).$$
The description of the kernel is due to the fact that $(N\times N)\cap Z$ consists
of those $(n, n, 1)$ with $\Inn(n)=1$.

\medskip \noindent (iv)
Suppose $(n_1, n_2, m)\in H$ commutes with $N_1$ modulo $Z$. Thus, for every $x\in N_1$ there
is $(\nu, \nu, \mu)\in Z$ with
$$(n_1 \alpha(m)(x), n_2, m)\ =\ (x n_1 \alpha(m)(\nu), n_2 \alpha(m)(\nu), m\mu).$$
The last two coordinates show that $\mu$ and $\nu$ are trivial. It follows that
$\alpha(m)(x) = n_1\inv x n_1$. Thus $\alpha(m)=\Inn(n_1)\inv$ and hence $(n_1, n_2, m)$
belongs to $N_2.Z$. The statement follows by symmetry and using the description of
the kernel of $N\times N\to G$ obtained above.

\medskip \noindent (v)
We can assume that $N$ has trivial centre.
Notice that the involutory automorphism of $N\times N$ defined by  $(u, v)
\mapsto (v,u)$ extends to a well defined automorphism of $H$ which
descends to an automorphism $\zeta$ of $G$ swapping the two factors
of the product  $N_1 . N_2$.

Let now $C < G$ be a (topologically) characteristic closed subgroup.
Assume first that $C \cap N_1 = 1$. Then $1 = \zeta(C \cap N_1) = C
\cap \zeta(N_1) = C \cap N_2$. Thus $C$ centralises $N_1. N_2$ and
is thus contained in $\centra(G)$ since $N_1. N_2$ is dense.
Therefore we have $C= 1$ by (iv) in this case.

Assume now that $C \cap N_1 \neq 1$. Then $N_1$ is contained in $C$
since $N$ is topologically simple by hypothesis. Transforming by the
involutory automorphism $\zeta$ shows that $N_2$, and hence also
$N_1.N_2$, is then contained in $C$, which implies that $C= G$ since
$C$ is closed and $N_1. N_2$ is dense. Thus $G$ is indeed
characteristically simple, as desired.

The above arguments show moreover that  $N_i$ are minimal closed
normal subgroups of $G$. This implies that if $G$ splits as a direct
product $G \cong L_1 \times L_2$ of closed normal subgroups, then,
upon renaming the factors, we have $N_i< L_i$ for $i=1, 2$. It
follows that $L_i < \centra_G(N_{3-i}) = N_i$ by (ii). Thus $N_i =
L_i$. In view of (iii), this implies that $G$ does not split
non-trivially as a direct product of closed subgroups provided
$\alpha(M)\neq \Inn(N)$.
\end{proof}

Our goal is now to present some concrete situations with $M$ and $N$ locally compact.

\begin{example}\label{ex:0}%
Let $M, N$ be totally disconnected locally compact groups and let  $\varphi:N\to M$\footnote{%
Of course $\varphi(N)$ is only analytic when $N$ is metrisable,
but this is the standard situation to which the subsection title refers.}
be a continuous injective homomorphism whose image is dense and normal in $M$. In particular, the conjugation action of $M$ on $\varphi(N)$ induces a homormorphism $\alpha : M \to \Aut(N)$; however $\alpha$ need not be continuous in general. However $\alpha$ is indeed continuous in the following cases.
\begin{itemize}
\item $N$ is discrete and $\varphi(N) \leq \QZ(M)$. 
\item $M = \overline{\Ad(N)}$ and $\varphi = \Ad$.  
\end{itemize}
Of course these two cases are not mutually exclusive. One is then in a position to invoke Proposition~\ref{prop:QuasiDirect:constrution}, which provides a totally disconnected locally compact group $G$ that is a quasi-direct product with two copies of $N$ as quasi-factors. 
\end{example}

It is now easy to construct non-trivial quasi-products of totally
disconnected groups by exhibiting a group
$N$ satisfying the required conditions. 

\begin{example}\label{ex:1bis}
Let $N$ be one of the discrete groups described in Example~\ref{ex:1}. This example yields a locally compact completion $M$ and a continuous homomorphism $\alpha : M \to \Aut(N)$ such that $\Ad(N) \leq \alpha(M) \leq \overline{\Ad(N)}$. If $N$ is simple, the group $G$ provided by Proposition~\ref{prop:QuasiDirect:constrution} is characteristically simple. In this way, we obtain various examples of characteristically simple locally compact groups which are quasi-products but do not split as direct products. Notice however that in these examples $N$ is not finitely generated and the corresponding $G$ is never compactly generated. 
\end{example}

Another way to satisfy the conditions of Example~\ref{ex:0} is to start with a group $N$ which is not \Adc, but whose adjoint closure $M = \overline{\Inn(N)}$ is locally compact. We proceed to describe a concrete example. Part of the interest of the example is that $N$ will be compactly generated, which implies that the associated groups $M$ and $G$ will be both compactly generated. Indeed, consider  a compact generating set 
$\Sigma$ for $N$ and $U$ any compact open subgroup of $M$. Then $U \cup \varphi(\Sigma)$
generates $M$, since $\la U \cup \varphi(\Sigma) \ra$ is open and
contains a dense subgroup. Thus $M$ is compactly generated, and so
is $H$ as well as all its quotients, including $G$.

\begin{example}\label{ex:2}
Consider the semi-direct product
$$N = \SL_3\big(\mathbf{F}_p(\!(t )\!)\big) \rtimes \ZZ,$$
where the cyclic group $\ZZ $ is any infinite cyclic subgroup of the
Galois group $\Aut\big(\mathbf{F}_p(\!(t )\!)\big)$. Then $N$ is not
\Adc, but $\Aut(N)$ is locally compact, as follows from
Theorem~\ref{thm:AutG:lc}. Moreover, the cyclic group $\ZZ$
normalises every characteristic subgroup of the compact open
subgroup
$$ \SL_3\big(\mathbf{F}_p[\![t ]\!]\big) <  \SL_3\big(\mathbf{F}_p(\!(t
)\!)\big),$$ from which it easily follows that $\ZZ$ contains an
unbounded asymptotically central sequence.

Notice furthermore that  $N$ is centrefree. The easiest way to see
this is by noticing that $N$ acts minimally without fixed point at
infinity on the Bruhat--Tits building associated with
$\SL_3\big(\mathbf{F}_p(\!(t )\!)\big)$ (see
\cite[Theorem~1.10]{Caprace-Monod_structure}). Thus
Proposition~\ref{prop:QuasiDirect:constrution} may be applied.
In conclusion, we find that the group
$$\frac{\Bigg(\Big(\SL_3\big(\mathbf{F}_p(\!(t )\!)\big) \rtimes
\ZZ\Big) \times \Big(\SL_3\big(\mathbf{F}_p(\!(t )\!)\big) \rtimes
\ZZ\Big)\Bigg)%
\rtimes \overline{\Ad\Big(\SL_3\big(\mathbf{F}_p(\!(t )\!)\big) \rtimes
\ZZ\Big)}}%
{\Big\{ (z, z, \Ad(z)\inv) : z\in \SL_3\big(\mathbf{F}_p(\!(t )\!)\big)
\rtimes \ZZ\Big\}}$$
provides an example of a compactly generated totally disconnected locally
compact group with trivial quasi-centre (in particular centrefree) which
is a non-trivial
quasi-product. However, this example is not characteristically simple.
\end{example}

\subsection*{Open problems}

Corollary~\ref{cor:CharSimple} naturally suggests the following
question.

\begin{question}\label{question:1}
Is there a compactly generated characteristically simple locally
compact group $G$ that is a quasi-product with at least two simple
quasi-factors, but which does not split non-trivially as a direct
product?
\end{question}

Finally, we mention two related problems.

\begin{question}\label{question:2}
Is there a compactly generated topologically simple locally compact
group $G$ which contains a proper dense normal subgroup?
\end{question}

\begin{question}\label{question:3}
Is there a compactly generated topologically simple locally compact
group $G$ which is not \Adc?
\end{question}

As explained in Example~\ref{ex:2}, a positive answer to the latter
implies a positive answer to both Questions~\ref{question:1}
and~\ref{question:2}. Moreover, Proposition~\ref{prop:MinMax} implies
that a positive answer to Question~\ref{question:1} also gives a
positive answer to Question~\ref{question:2}.

\subsection*{On dense normal subgroups of topologically simple groups}%
We close this appendix with the following result, due to
N.~Nikolov~\cite[Proposition~2]{Nikolov}, which provides however severe
restrictions on possible non-Hausdorff quotients of topologically
simple groups.

\begin{prop}\label{prop:Nikolov}
Let $M$ be a compactly generated totally disconnected locally
compact group which is locally finitely
generated. Assume that $M$ has no non-trivial compact normal subgroup.
Then any dense normal subgroup contains the derived group
$[M, M]$. In particular, if $M$ is topologically simple, then it is
abstractly simple if and only if it
is abstractly perfect.
\end{prop}

\begin{proof}
See~\cite[Proposition~2]{Nikolov}. The statement from \emph{loc.~cit.}
requires $M$ to be topologically simple, but only the absence of \emph{compact}
normal subgroups is used in the proof.
\end{proof}

\bibliographystyle{amsalpha}
\bibliography{../IsomCAT0}
\printindex
\end{document}